\documentclass[10pt]{amsart}
\usepackage{amsxtra,yhmath, amsfonts, amsmath, amsthm, amstext, amssymb, amscd, mathrsfs, verbatim, color}
\usepackage{threeparttable}
\theoremstyle{plain}
\newtheorem{theorem}{Theorem}[section]
\newtheorem{lemma}[theorem]{Lemma}
\newtheorem{definition-theorem}[theorem]{Definition-Theorem}
\newtheorem{proposition}[theorem]{Proposition}
\newtheorem{corollary}[theorem]{Corollary}

\newtheorem{notation}[theorem]{Notation}

\theoremstyle{definition}
\newtheorem{definition}[theorem]{Definition}
\newtheorem{example}[theorem]{Example}
\newtheorem{remark}[theorem]{Remark}

\newcommand \bth[1] { \begin{theorem}\label{t#1} }
\newcommand \ble[1] { \begin{lemma}\label{l#1} }

\newcommand \bpr[1] { \begin{proposition}\label{p#1} }
\newcommand \bco[1] { \begin{corollary}\label{c#1} }
\newcommand \bde[1] { \begin{definition}\label{d#1}\rm }
\newcommand \bex[1] { \begin{example}\label{e#1}\rm }
\newcommand \bre[1] { \begin{remark}\label{r#1}\rm }

\newcommand \bnota[1] {\begin{notation}\label{n#1}\rm }
\newcommand {\ele} { \end{lemma} }

\newcommand {\epr} { \end{proposition} }
\newcommand {\eco} { \end{corollary} }
\newcommand {\ede} { \end{definition} }
\newcommand {\eex} { \end{example} }
\newcommand {\ere} { \end{remark} }
\newcommand {\enota} { \end{notation} }

\DeclareMathOperator \ad { {\mathrm{ad}} }  

\DeclareMathOperator \tr { {\mathrm{tr}} }

\DeclareMathOperator \Hom { {\mathrm{Hom}} }

\DeclareMathOperator \ind{ {\mathrm{ind}}}
\DeclareMathOperator \Ind { {\mathrm{Ind}} }
\DeclareMathOperator \Irr { {\mathrm{Irr}} }

\DeclareMathOperator \gen { { \mathrm{gen}}}
\DeclareMathOperator \sgn { { \mathrm{sgn}}}
\DeclareMathOperator \trivial { {{\mathrm triv}}} 
\DeclareMathOperator \ellip  { { {\mathrm ell}}}
\DeclareMathOperator \im { { {\mathrm im}}}

\DeclareMathOperator \Pin { { {\mathrm Pin}}}
\DeclareMathOperator \Alt { { {\mathrm Alt}}}
\DeclareMathOperator \sol { { {\mathrm sol}}}

\renewcommand{\qed}{\begin{flushright} {\bf Q.E.D.}\ \ \ \ \
                  \end{flushright} }    
\input xy
\xyoption{all}

\begin{document}
%\tableofcontents
\setlength{\baselineskip}{1.2\baselineskip}
%%%%%%%%%%%%%%%%%%%%%%%%%%%%%%%%%%%%%%%%%%%%%%%%%%%%%%%%%%%%%%%%%%%%%%%%%%%
%%%%%%%%%%%%%%%%%%%%%%    Title    %%%%%%%%%%%%%%%%%%%%%%%%%%%%%%%%%%%%%%%%
\title[Spin representations of real reflection groups]
{Spin representations of real reflection groups of non-crystallographic root systems}
\author[Kei Yuen Chan]{Kei Yuen Chan}
\address{
Department of Mathematics \\
University of Utah}
\email{chan@math.utah.edu}

\maketitle

\begin{abstract}
A uniform parametrization for the irreducible spin representations of Weyl groups in terms of nilpotent orbits is recently achieved by Ciubotaru (2011). This paper is a generalization of this result to other real reflection groups. 

Let $(V_0, R, V_0^{\vee}, R^{\vee})$ be a root system with the real reflection group $W$. We define a special subset of points in $V_0^{\vee}$ which will be called solvable points. Those solvable points, in the case $R$ crystallographic, correspond to the nilpotent orbits whose elements have a solvable centralizer in the corresponding Lie algebra. Then a connection between the irreducible spin representations of $W$ and those solvable points in $V_0^{\vee}$ is established.

%We define solvable points in $V_0^{\vee}$ which, in the case of $R$ crystallographic, correspond to the nilpotent elements with a solvable centralizer in the corresponding Lie algebra. Then for noncrystallographic root system $R$, a surjective map f

\end{abstract}

\section{Introduction}

\subsection{Introduction} \label{s intro}
Let $(V_0, R, V_0^{\vee}, R^{\vee})$ be a reduced root system (see section \ref{ss def root sys}) with the corresponding real reflection group $W$. Let $O(V_0^{\vee})$ be the orthogonal group of $V_0^{\vee}$ and let $\Pin(V_0^{\vee})$ be the double cover of $O(V_0^{\vee})$ with the covering map $p$ (see Section \ref{ss clifford}). Since $W$ is a subgroup of orthogonal transformations on $V_0^{\vee}$, we may consider the preimage $p^{-1}(W)$, the double cover $\widetilde{W}$ of $W$. The classification of genuine irreducible representations of $\widetilde{W}$ (i.e. representations which do not factor through $W$), or so-called spin representations of $W$, has been known for a long time from the work of Schur, Morris, Read, Stembridge and others (\cite{Mo1}, \cite{Mo2}, \cite{Re}, \cite{Re2} and \cite{St}). A recent paper \cite{Ci} by Ciubotaru introduces an approach to classify genuine representations in the framework of nilpotent orbits in semisimple Lie algebras. The orbits of nilpotent elements whose elements have a solvable centralizer play an essential role in the classification. This paper generalizes \cite{Ci} to the reflection groups associated to noncrystallographic root systems.

An important ingredient in our paper is a special subset of points in $V_0^{\vee}$, which will be called solvable points. For maximum generality, fix a $W$-invariant parameter function $c: R \rightarrow \mathbb{R}$ and write $c_{\alpha}$ for $c(\alpha)$. 
Before giving a precise definition of solvable points, we define a point $\gamma \in V_0^{\vee}$ to be {\it distinguished} if 
\[   | \left\{ \alpha \in R : (\alpha, \gamma)=c_{\alpha} \right\}| =|\left\{ \alpha \in R : (\alpha, \gamma)=0 \right\}|+\dim_{\mathbb{R}} V_0 .\]
 This combinatorial definition is introduced in \cite{HO} by Heckman and Opdam and can be viewed as a generalization of the characterization of distinguished nilpotent orbits in the Bala-Carter theory (\cite{Ca}). Let $\Delta$ be a fixed set of simple roots in $R$. For $J\subseteq \Delta$, let $V_{0,J}^{\vee}$ be the real vector space spanned by coroots corresponding to $J$.  We extend the set of distinguished points to a larger class:
\begin{definition}
A point $\gamma$ in $V_0^{\vee}$ is said to be {\it solvable} if $\gamma$ is conjugate to some point which is distinguished in $V_{0,J}^{\vee}$ for some $J \subseteq \Delta$, and satisfies
\[   | \left\{ \alpha \in R : (\alpha, \gamma)=c_{\alpha} \right\}| =|\left\{ \alpha \in R : (\alpha, \gamma)=0 \right\}|+|J| .\]
Let $\mathcal V_{\mathrm{sol}}$ be the set of $W$-orbits of solvable points in $V^{\vee}_0$. 
\end{definition}
\noindent
It is easy to see that a solvable point with $J=\Delta$ is distinguished. The set of solvable points also contains the set of quasi-distinguished points in the sense of \cite{Ree} by Reeder. The terminology of solvable points is explained by the following result, which will be proven in Section \ref{s def sol}:

\begin{theorem} \label{thm def sol int}
Let $R$ be a crystallographic root system and set $c \equiv 2$. Let $\mathfrak{g}$ be the semisimple Lie algebra associated to $R$. Let $\mathcal N_{\sol}$ be the set of nilpotent orbits in $\mathfrak{g}$ whose elements have a solvable centralizer. Then there is a natural one-to-one correspondence between the set $\mathcal V_{\sol}$ and the set $\mathcal N_{\sol}$. 
\end{theorem}
\noindent
Roughly, the correspondence in Theorem \ref{thm def sol int} takes a nilpotent orbit to the semisimple element of a corresponding Jacobson-Morozov triple. 

%Due to the absence of an underlying group for noncrystallographic cases, we need to find an appropriate substitute for nilpotent orbits. Fortunately, Kazhdan and Lusztig (\cite{KL}) established a connection between the nilpotent orbits and the tempered modules of the graded affine Hecke algebra associated to a root system. The later object can be defined algebraically in the non-crystallographic cases and has been studied by several authors (see \cite{HO}, \cite{KR} and \cite{Kr}). Thus as suggested in the paper \cite{Ci} of Ciubotaru, tempered modules should be a suitable replacement. More precisely, in this paper, we build a surjective map from the irreducible genuine $\widetilde{W}$ representations to some points in $V^{\vee}$ which are expected to be the central characters of certain tempered modules. Those points will be called solvable points. The terminology is explained by the property that for $R$ crystallographic, it corresponds to the nilpotent orbits whose elements have solvable centralizers in the Lie algebra.

%In particular, when a point is solvable with $J$ to be the entire set of simple roots, the point is distinguished.
Our main result Theorem \ref{thm main result} below establishes a connection between those solvable points and the spin representations of $W$ in the case that $R$ is noncrystallographic. Before stating the main result, we need few more notations. Fix a symmetric $W$-invariant bilinear form $\langle, \rangle$ on $V^{\vee}_0$ as in (\ref{notation inner}). Define a Casmir-type element in the Clifford algebra $C(V_0^{\vee})$:
\[ \Omega_{\widetilde{W}} =- \frac{1}{4} \sum_{\alpha>0, \beta>0,s_{\alpha}(\beta)<0} c_{\alpha}c_{\beta} | \alpha^{\vee}||\beta^{\vee}| f_{\alpha}f_{\beta}  ,\]
where $s_{\alpha} \in W$ is the reflection corresponding to $\alpha$ and $f_{\alpha}$ is a certain element in $p^{-1}(s_{\alpha})$ (see Section \ref{ss clifford}) and $|\alpha^{\vee}|=\langle \alpha^{\vee}, \alpha^{\vee} \rangle^{1/2}$. The element $\Omega_{\widetilde{W}}$ is introduced in \cite{Ci} and is related to the Dirac operator for the graded affine Hecke algebra in \cite{BCT} by Barbasch, Ciubotaru and Trapa. Indeed, $\Omega_{\widetilde{W}}$ is in the center of $C(V_0^{\vee})$ and so a version of the Schur's lemma implies that this element acts on a simple $\widetilde{W}$-representation $(\widetilde{U}, \widetilde{\chi})$ by a scalar $\widetilde{\chi}(\Omega_{\widetilde{W}})$. Let $\Irr_{\mathrm{gen}}(\widetilde{W})$ be the set of irreducible genuine $\widetilde{W}$-representations. Define an equivalence relation $\sim$ on $\Irr_{\gen}(\widetilde{W})$: $\widetilde{\sigma} \sim \widetilde{\sigma} \otimes \sgn$, where $\sgn$ is the sign $W$-representation.

%For a simple $\widetilde{W}$-module $(\widetilde{U},\widetilde{\chi})$, we may naturally regard $\widetilde{U}$ as a $W$-module with the character $\sgn \otimes \widetilde{\chi}$. Here $\sgn$ is the sign $W$-representation. Then $\Omega_{\widetilde{W}}$ sends $(\widetilde{U},\widetilde{\chi})$ to $(\widetilde{U},\sgn \otimes \widetilde{\chi})$ by multiplying a scalar, denoted by $\widetilde{\chi}(\Omega_{\widetilde{W}})$. 

%Indeed, $\Omega_{\widetilde{W}}$ is in the center of $\mathbb{C}[\widetilde{W}]$ and so a version of the Schur's lemma implies that this element acts on a simple $\widetilde{W}$-module $(\widetilde{U}, \widetilde{\chi})$ by a scalar $\widetilde{\chi}(\Omega_{\widetilde{W}})$. Let $\Irr_{\mathrm{gen}}(\widetilde{W})$ be the set of irreducible genuine $\widetilde{W}$-representations. Define an equivalence relation $\sim$ on $\Irr_{\gen}(\widetilde{W})$: $\widetilde{\sigma} \sim \widetilde{\sigma} \otimes \sgn$.

In Section \ref{s surj}, we prove our main result:
\begin{theorem} \label{thm main result}
Let $(V_0, R, V_0^{\vee}, R^{\vee})$ be a noncrystallographic root system. Fix a symmetric $W$-invariant bilinear form $\langle , \rangle$ on $V_0^{\vee}$ as in (\ref{notation inner}). Then there exists a unique surjective map 
\[\Phi: \Irr_{\gen}(\widetilde{W})/\sim \rightarrow \mathcal V_{\mathrm{sol}}\]
 such that for any representative $\gamma \in \Phi([\widetilde{\chi}])$,
\begin{eqnarray} \label{eq surj map cond}
 \widetilde{\chi}(\Omega_{\widetilde{W}})=\langle \gamma,\gamma \rangle .
\end{eqnarray}
Furthermore, $\Phi$ is bijective if and only if $c_{\alpha} \neq 0$ for some $\alpha \in R$, and either:
\begin{enumerate}
\item[(1)] $R=I_2(n)$ ($n$ odd) or $H_3$; or
\item[(2)] $R=I_2(n)$ ($n$ even) with  $\cos(k\pi/n) c'-\cos(l\pi/n) c'' \neq 0$
for any integers $k,l$ which have distinct parity and $\cos(k\pi/n), \cos(l\pi/n)\neq \pm 1$, where $c'$ and $c''$ are the two values corresponding to the two distinct $W$-orbits in the parameter function $c$ . 
\end{enumerate}\end{theorem}
\noindent
We see from the expression of $\Omega_{\widetilde{W}}$ that $(\sgn \otimes \widetilde{\chi})(\Omega_{\widetilde{W}})=\widetilde{\chi}(\Omega_{\widetilde{W}})$ and this explains why we need the equivalence $\sim$ in defining the surjective map above. Although the map $\Phi$ in $H_4$ or some special cases for $I_2(\mathrm{even})$ fails to be bijective, the sizes of the fibers of $\Phi$ are still one most of time. 

For the crystallographic cases (with the equal parameter $c \equiv 2$), we could still obtain a surjective map as the one in Theorem \ref{thm main result} by using our Theorem \ref{thm def sol int} to replace the image of the surjective map in \cite[Theorem 1]{Ci} with the set of solvable points. This explains how our result generalizes \cite{Ci}. However, a surjective map with only the property (\ref{eq surj map cond}) is not unique in general in the crystallographic cases. 
%This explains how our result generalizes }.

%However, the condition of the equality in Theorem \ref{thm main result} is not strong to guarantee the uniqueness in some cases.

Our motivation for the Theorem \ref{thm main result} is to study the Dirac cohomology for the graded affine Hecke algebra in the noncrystallographic cases by analogue with the crystallographic cases in \cite{BCT} and \cite{CT}. We expect that those solvable points afford the central characters of interesting tempered modules with nonzero Dirac cohomology (in the sense of \cite{BCT}). Theorem \ref{thm main result} is evidence for the claim, but some further information such as the $W$-module structures of tempered modules is needed. In view of the paper \cite{CT} by Ciubotaru and Trapa, perhaps a preliminary step to understand the Dirac cohomology is to look at the elliptic representation theory of $W$ (\cite{Re}). A study of the latter object is carried out in Section \ref{s ellip}.

%\begin{conjecture}
%For each $\widetilde{\chi} \in \Irr_{\mathrm{gen}}(\widetilde{W})$, there exists a tempered module $X$ with the central character $\Phi([\widetilde{\chi}])$ such that $\Hom_{\widetilde{W}}(\widetilde{\sigma}, X \otimes S) \neq 0$, where $S$ is the basic spin module of $\widetilde{W}$ (see section \ref{ss clifford}).
%\end{conjecture}

%Theorem \ref{thm main result} is a weaker generalization of the surjective map in \cite[Theorem 1]{Ci} since it does not really relate to the $W$-module structure of the tempered modules. Nonetheless we may conjecture that for each $\widetilde{\sigma} \in \Irr_{\mathrm{gen}}(\widetilde{W})$, there exists a tempered module $X$ with the real central character $\Phi([\widetilde{\sigma}])$ such that $\Hom_{\widetilde{W}}(\widetilde{\sigma}, X \otimes S) \neq 0$, where $S$ is the basic spin module of $\widetilde{W}$ (see section \ref{ss clifford}).

\subsection{Acknowledgment}
The author would like to thank Dan Ciubotaru and Peter Trapa for suggesting this topic. He is also very grateful to them providing guidances, answering questions and pointing out references during his work. The author would also like to thank the referee for useful comments.

\section{Preliminaries}

\subsection{Root systems and notations} \label{ss def root sys}

Fix a reduced root system $\Sigma=(V_0, R, V_0^{\vee}, R^{\vee})$ over $\mathbb{R}$ such that
\begin{itemize}
\item[(1)] $R \subset V_0 $ and $R^{\vee} \subset V_0^{\vee} $ span the real vector spaces $V_0$ and $V_0^{\vee}$ respectively.
\item[(2)] There exists a bilinear pairing 
\[ (.,.): V_0 \times V_0^{\vee} \rightarrow \mathbb{R}, \]
and a bijection from $R$ to $R^{\vee}$, denoted $\alpha \mapsto \alpha^{\vee}$, such that $(\alpha, \alpha^{\vee})=2$ for any $\alpha \in R$.
\item[(4)] For $\alpha \in R$, the reflections
  \[ s_{\alpha}: V_0 \rightarrow V_0,\quad s_{\alpha}(v)=v-(v,\alpha^{\vee})\alpha, \]
  \[ s_{\alpha}^{\vee}:V_0^{\vee} \rightarrow V_0^{\vee}, \quad s_{\alpha}^{\vee}(v')=v'-(\alpha,v')\alpha^{\vee} \]
 leave $R$ and $R^{\vee}$ invariant respectively. 
 \item[(5)] For $\alpha \in R$, the only multiples of $\alpha$ in $R$ are $\alpha$ and $-\alpha$. Moreover, $0 \not\in R$ and $0 \not\in R^{\vee}$.
\end{itemize}

A root system $\Sigma$ (or simply $R$) is said to be {\it crystallographic} if $(\alpha, \beta^{\vee}) \in \mathbb{Z}$ for all $\alpha, \beta \in R$. In the literature, a root system is often by definition crystallographic and our terminology of root systems here is not quite standard. 

Our primary concern in this paper is noncrystallographic cases. This includes $I_2(n)$ for $n=5$ and $n \geq 7$, $H_3$ and $H_4$ and their corresponding Dynkin diagrams are as follows respectively:
\[ 
\xy
(-70,4)*{\alpha_1};(-60,4)*{\alpha_2}; (-65,-2)*{n} ;(-70,0)*{\circ}="1";(-60,0)*{\circ}="2";
{\ar@{-} "1"; "2"  }; 
(-50,4)*{\alpha_1};(-40,4)*{\alpha_2}; (-45,-2)*{5} ;(-50,0)*{\circ}="4";(-40,0)*{\circ}="5";
{\ar@{-} "4"; "5"  };(-30,0)*{\circ}="6";{\ar@{-} "5";"6" }; (-30,4)*{\alpha_3};
(-20,4)*{\alpha_1};(-10,4)*{\alpha_2}; (-15,-2)*{5} ; (-20,0)*{\circ}="8";(-10,0)*{\circ}="9"; {\ar@{-} "8";"9" }; (0,0)*{\circ}="10";{\ar@{-} "9";"10" };  (10,0)*{\circ}="11";  {\ar@{-} "10"; "11"}; (10,4)*{\alpha_4};(0,4)*{\alpha_3};
\endxy
\]
When $I_2(n)$ is mentioned later, we do not necessarily assume $I_2(n)$ to be noncrystallographic i.e. $n$ can be any integer greater than or equal to $3$. Our results naturally cover all $I_2(n)$. When $R=I_2(n), H_3, H_4$, we shall fix a $W$-invariant inner product $\langle.,.\rangle$ on $V^{\vee}_0$ such that
\begin{eqnarray} \label{notation inner} 
\langle \alpha^{\vee}, \alpha^{\vee} \rangle &=&2 
\end{eqnarray}
for all $\alpha^{\vee} \in R^{\vee}$. Let $V=\mathbb{C} \otimes_{\mathbb{R}} V_0$ and let $V^{\vee}=\mathbb{C} \otimes_{\mathbb{R}} V_0^{\vee}$. Then we extend $\langle .,. \rangle$ to a symmetric $W$-invariant $\mathbb{C}$-bilinear form on $V^{\vee}$.

Denote by $W(R)$ or simply $W$ the subgroup of $GL(V_0)$ generated by all the reflections $s_{\alpha}$, where $\alpha \in R$. The map $s_{\alpha} \mapsto s_{\alpha}^{\vee}$ gives an embedding of $W$ into $GL(V_0^{\vee})$ so that 
$  (v,w\cdot\omega)=(w\cdot v, \omega) $ for all $w \in W$, $v \in V_0$ and  $\omega \in V_0^{\vee}$. 

 Fix a set $R_+$ of positive roots in $R$. Write $\alpha >0$ for $\alpha \in R_+$ and write $\alpha<0$ for $\alpha \in R \setminus R_+$. Let $R^{\vee}_+=\left\{ \alpha^{\vee} : \alpha \in R_+ \right\}$. Let $\Delta=\left\{ \alpha_1, \ldots, \alpha_r \right\}$ be the set of simple roots in $R_+$, where $r=\dim_{\mathbb{R}} V_0$.

The Coxeter group $W(I_2(n))$ is the dihedral group of order $2n$. The generators $s_{\alpha_1}$, $s_{\alpha_2}$ are subject to the relation: $(s_{\alpha_1}s_{\alpha_2})^n=1$. The Coxeter group $W(H_3)$ has order $120$ and is generated by 3 simple reflections, namely $s_{\alpha_1}, s_{\alpha_2}, s_{\alpha_3}$, subject to the following relations:
\[(s_{\alpha_1}s_{\alpha_2})^5=(s_{\alpha_2}s_{\alpha_3})^3=(s_{\alpha_1}s_{\alpha_3})^2=1. \]
The Coxeter group $W(H_4)$ has order $14400$ and is generated by $4$ simple reflections, $s_{\alpha_1}, s_{\alpha_2}, s_{\alpha_3}, s_{\alpha_4}$ with the following relations:
\[ (s_{\alpha_1}s_{\alpha_2})^5=(s_{\alpha_2}s_{\alpha_3})^3=(s_{\alpha_3}s_{\alpha_4})^3=1\]
and 
\[ (s_{\alpha_i}s_{\alpha_j})^2=1 \ \mbox{for $|i-j|>1$} . \]

%%% subsection of the graded affine Hecke algebra
%\subsection{The graded affine Hecke algebra}

%A graded Hecke algebra $\mathbb{H}$ associated to a root system $\Phi$ is the complex associative algebra generated by the symbols $\left\{ t_w : w \in W \right\}$ and $\left\{ t_f : f \in S(V^{\vee}) \right\}$ subject to the following relation:
%\begin{itemize}
%\item[(1)] The linear map from the group algebra $\mathbb{C}[W]$ to $\mathbb{H}$ sending $w$ to $t_w$ is an injective map of algebra.
%\item[(2)] The linear map from the group algebra $S(V^{\vee})$ to $\mathbb{H}$ taking $f$ to $t_f$ is also an injective map of algebra.\\
%\indent We shall later view $\mathbb{C}[W]$ and $S(V^{\vee})$ as subalgebras of $\mathbb{H}$. We shall write $f$ instead of $t_f$ in $\mathbb{H}$ for $f \in S(V^{\vee})$. For $\alpha \in R$, we shall briefly write $t_{\alpha}$ for $t_{s_{\alpha}}$.
%\item[(3)] The final relation is $\omega t_{\alpha}-t_{\alpha}s_{\alpha}(\omega)=c_{\alpha}(\alpha, \omega)$ for $\omega \in V^{\vee}$ and $\alpha \in R$. Here $s_{\alpha}(\omega)$ is the element in $V^{\vee}$ obtained by $s_{\alpha}$ acting on $\omega$. 
%\end{itemize}

%%%%%%%%%%%%%%%

\subsection{The Clifford algebra and the Pin group} \label{ss clifford}
Fix an inner product $\langle , \rangle$ on $V_0^{\vee}$ as in (\ref{notation inner}). Let $T(V_0^{\vee})$ be the tensor algebra of $V_0^{\vee}$. Let $I$ be the ideal generated by all the elements of the form
\[   \omega \otimes \omega'+\omega'\otimes \omega+2\langle \omega, \omega' \rangle . \]
Then the real Clifford algebra, denoted $C(V_0^{\vee})$, with respect to $V_0^{\vee}$ and the inner product $\langle , \rangle$ is the quotient algebra $T(V_0)/I$. 

Let $^t$ be the anti-automorphism of $C(V_0^{\vee})$ characterized by the properties:
\[ (ab)^t=b^ta^t \mbox{ for $a,b \in C(V_0^{\vee})$}, \omega^t=-\omega, \mbox{ for $\omega \in V_0^{\vee}$}. \]
Define a map $\epsilon: V_0^{\vee} \rightarrow V_0^{\vee}$, $\epsilon(\omega)=-\omega$. This induces an automorphism, still denoted by $\epsilon$, from $C(V_0^{\vee})$ to $C(V_0^{\vee})$. 
%Let $O(V_0^{\vee})$ be the orthogonal transformation on $V_0^{\vee}$ with respect to $\langle , \rangle$. 
The map $\epsilon$ decomposes $C(V_0^{\vee})$ into $+1$ eigenspace $C(V_0^{\vee})_{\mathrm{even}}$ and $-1$ eigenspace $C(V_0^{\vee})_{\mathrm{odd}}$: 
\[ C(V_0^{\vee})=C(V_0^{\vee})_{\mathrm{even}}\oplus C(V^{\vee}_0)_{\mathrm{odd}} .\]
This induces a $\mathbb{Z}_2$-grading on $C(V_0^{\vee})$.

%Define an automorphism $\epsilon: C(V_0^{\vee}) \rightarrow C(V_0^{\vee})$ such that 
%\[     \epsilon(a) = \left\{ \begin{array}{c c}
%                                   a & \mbox{ if $a \in C_{\mathrm{even}}(V^{\vee}_0)$}  \\
%                                   -a & \mbox{ if $a \in C_{\mathrm{odd}} (V^{\vee}_0)$} 
%                     \end{array}          \right. .
%\]
Define the Pin group of $V_0^{\vee}$ to be
\[ \mathrm{Pin}(V_0^{\vee}):= \left\{ a \in C(V_0^{\vee})   :      \epsilon(a)V_0^{\vee}a^{-1} \subset V_0^{\vee}, a^t=a^{-1}     \right\}.\]
We get a surjective homomorphism $p$ from $\Pin(V_0^{\vee})$ to $O(V_0^{\vee})$ defined by 
\[  p(a)\omega=\epsilon(a)\omega a^{-1}.\]
Here $O(V_0^{\vee})$ is the orthogonal group of $V_0^{\vee}$ with respect to the inner product $\langle , \rangle $.
 
 Then we have the following exact sequence:
\[ 1 \rightarrow \left\{ \pm 1\right\} \rightarrow \mathrm{Pin}(V_0^{\vee}) \stackrel{p}{\rightarrow} O(V_0^{\vee}) \rightarrow 1.\]

Let $\det_{V_0^{\vee}}: O(V_0^{\vee}) \rightarrow \left\{ \pm1 \right\}$ be the determinant function on $O(V_0^{\vee})$. Regarding $W$ as a subgroup of $O(V_0^{\vee})$, let $W_{\mathrm{even}}=\ker(\det_{V_0^{\vee}}) \cap W$. For the notational convenience later, set
\[W' = \left\{ \begin{array} {l l}
                W      & \mbox{ if $\dim_{\mathbb{R}} V_0$ is odd} \\
                W_{\mathrm{even}} & \mbox{ if $\dim_{\mathbb{R}} V_0$ is even} 
              \end{array}
        \right. .
   \]
Define 
\begin{equation*}
 \widetilde{W}=p^{-1}(W) \quad \mbox{ and } \quad\widetilde{W}'=p^{-1}(W'). 
 \end{equation*}

%Further define
%\begin{eqnarray*}
% \widetilde{W}= p^{-1}(W) .
% \end{eqnarray*}

We describe the structure of $\widetilde{W}$. For $\alpha \in R$, define $f_{\alpha}$ to be the element in $C(V_0^{\vee})$ such that 
\[  f_{\alpha}=\frac{\alpha^{\vee}}{|\alpha^{\vee}|} .\]
Then one sees that $p(f_{\alpha})=s_{\alpha}$ and $f_{\alpha}^2=-1$. Furthermore,
\[ f_{\alpha}f_{\beta}f_{\alpha}=f_{s_{\alpha}(\beta)} .\]
The group $\widetilde{W}$ is generated by all $f_{\alpha}$ for $\alpha \in \Delta$.

\subsection{The double covers $\widetilde{W}$} \label{ss dc characters}
In this subsection, we review representations of $\widetilde{W}$ for non-crystallographic root systems. 

\bigskip

\noindent
{\it Case of $W=W(I_2(n))$}

The double cover $\widetilde{W}$ has generators $f_{\alpha_1}, f_{\alpha_2}$ (notation in Section \ref{ss clifford}) and is determined by the relations:
\[ (f_{\alpha_1}f_{\alpha_2})^n=f_{\alpha_1}^2=f_{\alpha_2}^2=-1 . \]

The character table (including genuine and nongenuine representations) of $\widetilde{I_2(n)}$ is provided in Table \ref{tb ch odd} and \ref{tb ch even} for the completeness.

\begin{table}[!ht]
  \caption{Character table of $\widetilde{W(I_2(n))}$ ($n$ odd) }  \label{tb ch odd}
\begin{tabular}{l c c c c c c }
\hline
 Characters     & $1$ & $-1$  &  $f_{\alpha_1}$ & $-f_{\alpha_1}$  & $(f_{\alpha_1}f_{\alpha_2})^k$ & $-(f_{\alpha_1}f_{\alpha_2})^k$   \\
                &   &          &                &                  & $(k=1, \ldots, \frac{n-1}{2})$       & $(k=1, \ldots,\frac{n-1}{2})$        \\
      \hline
$\trivial$     & $1$  & $1$   & $1$          &       $1$         &          $1$                   &   $1$                      \\  
$\sgn$          & $1$  & $1$   & $-1$         &       $-1$        &          $1$                   &   $1$                          \\  
$\phi_i$   ($i=1,\ldots,\frac{n-1}{2}$)       & $2$ & $2$ & $0$          &       $0$          &    $2\cos \frac{2ik\pi}{n}$ & $2\cos \frac{2ik\pi}{n}$ \\
$\widetilde{\chi}_1$ & $1$ & $-1$ & $-\sqrt{-1}$         &        $\sqrt{-1}$         &          $(-1)^k$       &   $-(-1)^k$         \\
$\widetilde{\chi}_2$ & $1$ & $-1$ & $\sqrt{-1}$         &        $-\sqrt{-1}$         &             $(-1)^k$       &   $-(-1)^k$        \\
$\widetilde{\rho}_i$ ($i=1,\ldots,\frac{n-1}{2}$) & $2$ & $-2$ & $0$  &      $0$          &    $2(-1)^{k}\cos \frac{2ik\pi}{n}$ & $2(-1)^{k+1}\cos \frac{2ik\pi}{n}$  \\
\hline
\end{tabular}

\end{table}

\begin{table} [!ht]
  \caption{Character table of $\widetilde{W(I_2(n))}$ ($n$ even) } \label{tb ch even}
\begin{tabular}{l c c c c c c c} 
\hline
Characters    & $1$ & $-1$  &  $f_{\alpha_1}$ & $f_{\alpha_2}$  & $(f_{\alpha_1}f_{\alpha_2})^k$ & $-(f_{\alpha_1}f_{\alpha_2})^k$ &  $(f_{\alpha_1}f_{\alpha_2})^{n/2}$  \\
              &     &       &                 &                  &$(k=1,\ldots,\frac{n-2}{2})$   & $(k=1,\ldots,\frac{n-2}{2})$  &         \\
      \hline
$\trivial$     & $1$  & $1$   & $1$          &       $1$         &          $1$                   &   $1$                      &  $1$ \\  
$\sgn$          & $1$  & $1$   & $-1$         &       $-1$        &          $1$                   &   $1$                      &   $1$     \\  
$\sigma_{\alpha_1}$ & $1$ & $1$ & $-1$         &        $1$         &          $(-1)^k$       &   $(-1)^k$   &      $(-1)^{n/2}$      \\
$\sigma_{\alpha_2}$ & $1$ & $1$ & $1$         &        $-1$         &          $(-1)^k$       &   $(-1)^k$   &      $(-1)^{n/2}$      \\
$\phi_i$   $(i=1,\ldots,\frac{n-2}{2})$        & $2$ & $2$ & $0$          &       $0$          &    $2\cos \frac{2ik\pi}{n}$ & $2\cos \frac{2ik\pi}{n}$ &  $2(-1)^i$ \\   
$\widetilde{\rho}_i$  $(i=1,\ldots,\frac{n}{2})$  & $2$ & $-2$ & $0$       &      $0$          &    $2\cos \frac{(2i-1)k\pi}{n}$ & $-2\cos \frac{(2i-1)k\pi}{n}$ &  0  \\
\hline
\end{tabular}
\end{table}

\bigskip
\noindent
{\it Case of $W=W(H_3)$}

It is well-known the group $W(H_3)$ can be realized as $\Alt_5 \times \mathbb{Z}_2$ such that the longest element $w_0$ of $W(H_3)$ is identified with the order $2$ element in $\mathbb{Z}_2$. Here $\Alt_5$ is the alternating group on $5$ letters. Since there are $5$ irreducible representations for $\Alt_5$, $W$ has $10$ irreducible representations.

We describe the group structure of the double cover $\widetilde{W(H_3)}$. Let $\widetilde{H}$ be the subgroup generated by $f_{\alpha_1}f_{\alpha_2}$ and $f_{\alpha_2}f_{\alpha_3}$. Then $\widetilde{H}$ is isomorphic to the double cover $\widetilde{\Alt}_5=p^{-1}(\Alt_5)$ of $\Alt_5$. % Using the group relations given in \cite[Chapter 2]{HH},
Moreover, for a lift $\widetilde{w}_0 \in \widetilde{W}$ of $w_0$, $\widetilde{w}_0$ is in the center of $\widetilde{W}$ and $\widetilde{w}_0^2=-1$. The double cover $\widetilde{W}$ is generated by $\widetilde{H}$ and $\widetilde{w}_0$.

%isomorphic to the group generated by elements of $\widetilde{\Alt}_5$ and an arbitrary order $4$ element, denoted $\zeta$, subject to the relations 
%\[ g\zeta=\zeta g\ \mbox{ for all $g \in \widetilde{\Alt}_5$}, \mbox{ and }  \zeta^2=z, \]
%where $z$ is in the center of  $\widetilde{\Alt}_5$ and has order $2$.
Under this identification, the representations of $W$ and $\widetilde{W}$ can be deduced from that of $\Alt_5$ and $\widetilde{\Alt}_5$ respectively. The latter two objects are well known (see, for example, in \cite[Table 1]{Re}). We provide the representations of $\Alt$ and $\widetilde{\Alt}_5$ in Table \ref{tb ch h3} and Table \ref{tb chg h3} for completeness. For each representation $\phi$ of $\Alt_5$, there are two corresponding representations of $W$, denoted by $\phi^+$ and $\phi^-$ such that $\phi^+$ and $\phi^-$ are determined by 
\[ \phi^+(h)=\phi^-(h)=\phi(h) \mbox{ for any $h \in H$ }, \quad \phi^{\pm}(w_0)=\pm 1 , \]
where $H=p(\widetilde{H})$ is identified with $\Alt_5$. 

Similarly, for each genuine representation $\widetilde{\chi}$ of $\widetilde{\Alt}_5$, denote by $\widetilde{\chi}^+$ and $\widetilde{\chi}^-$ the two representations of $\widetilde{W}$ corresponding to $\widetilde{\chi}$ such that 
\[ \widetilde{\chi}^+(h)=\widetilde{\chi}^-(h)=\widetilde{\chi}(h) \mbox{ for any $h \in \widetilde{H}$ }, \quad \widetilde{\chi}^{\pm}(\widetilde{w}_0)=\pm \sqrt{-1} .\]
In particular, there are $8$ genuine irreducible $\widetilde{W}$-representations. 

Let $\tau=(1+\sqrt{5})/2$ and $\bar{\tau}=(1-\sqrt{5})/2$, which will be used in the following Table \ref{tb ch h3} and Table \ref{tb chg h3}.

%(Let $W^{\mathrm{even}}$ be the subgroup of $W(H_3)$ containing elements of even length. Then a way to see $W^{\mathrm{even}}$ is isomorphic to $Alt_5$: $s_{\alpha_1}s_{\alpha_2} \mapsto (1,2,3,4,5)$, $s_{\alpha_1}s_{\alpha_3} \mapsto (1,2)(3,4)$ )

\begin{table}[!ht]
  \caption{Character table of $\Alt_5$ (identified with $H$) }  \label{tb ch h3} 
\begin{tabular}{c c c c c c }
\hline
 Characters     & $1$ & $s_{\alpha_1}s_{\alpha_2}$  &  $(s_{\alpha_1}s_{\alpha_2})^2$ &   $s_{\alpha_2}s_{\alpha_3}$  &     $s_{\alpha_1}s_{\alpha_3}$ \\
%                & $w_0$ & $w_0s_{\alpha_1}s_{\alpha_2}$ & $w_0(s_{\alpha_1}s_{\alpha_2})^2$ &    $w_0s_{\alpha_2}s_{\alpha_3}$   &     $w_0s_{\alpha_1}s_{\alpha_3}$         &             (2)\\
      \hline
$\phi_1$     & $1$  & $1$   & $1$          &       $1$         &          $1$                      \\  
$\phi_4$         & $4$  & $-1$   & $-1$         &       $1$        &          $0$                    \\  
$\phi_3$         & $3$ & $\tau$ & $\bar{\tau}$   &       $0$          &    $-1$             \\
$\overline{\phi}_3$ & $3$ & $\bar{\tau}$ & $\tau$  &    $0$         &     $-1$              \\
$\phi_5$& $5$ & $0$ & $0$         &     $-1$         &     $1$              \\
\hline
\end{tabular} \\
\end{table}

\begin{table}[!ht]
  \caption{Character table of genuine representations of $\widetilde{\Alt}_5$ (identified with $\widetilde{H})$ }  \label{tb chg h3}
  \begin{threeparttable}
\begin{tabular}{c c c c c c }
\hline
 Characters     & $\pm 1$ & $\pm f_{\alpha_1}f_{\alpha_2}$  &  $\pm (f_{\alpha_1}f_{\alpha_2})^2$ &   $\pm f_{\alpha_2}f_{\alpha_3}$  &     $f_{\alpha_1}f_{\alpha_3}$                          \\
%                & $-1$ & $-f_{\alpha_1}f_{\alpha_2}$ & $-(f_{\alpha_1}f_{\alpha_2})^2$ &  $-f_{\alpha_2}f_{\alpha_3}$ &           \\
 %     &             (2)\\
      \hline
$\widetilde{\chi}_2$    & $\pm 2$  & $\pm \bar{\tau}$   & $\mp \tau$     &       $\pm 1$        &          $0$                   \\  
$\widetilde{\overline{\chi}}_2$  & $\pm 2$ & $\pm \tau$ & $\mp \bar{\tau}$   &       $\pm 1$          &    $0$             \\
$\widetilde{\chi}_6$ & $\pm 6$ & $\mp 1$ & $\pm 1$  &    $0$         &     $0$          \\
$\widetilde{\chi}_4$ & $\pm 4$ & $\pm 1$ & $\mp 1$         &     $\mp 1$ &     $0$              \\
\hline
\end{tabular} 
%\begin{tablenotes}
%\item {\footnotesize For a representative $\widetilde{w}$ of a conjugacy class in row 1 (resp. row 2), $\widetilde{\chi}(\widetilde{w})$ is given by the value (resp. negative of the value) in the corresponding entry.         }
%\end{tablenotes}
\end{threeparttable}
\end{table}
\bigskip

\noindent
{\it Case of $W=W(H_4)$}

Recall the group structure of $\widetilde{W(H_4)}$ described in \cite{Re}. Let $G=\widetilde{\Alt}_5$, and let $G'=G \times G$. Let $\theta: G' \rightarrow G'$ be an automorphism of order $2$ such that $\theta(g_1,g_2)=(g_2,g_1)$. Then the double cover $\widetilde{W(H_4)}$ is isomorphic to the semidirect product $G' \rtimes \langle \theta \rangle$ and has order $28800$. Let $z$ be the unique nontrivial element in the center of $\widetilde{\Alt}_5$. The center $Z$ of $\widetilde{W(H_4)}=G' \rtimes \langle  \theta \rangle$ is $\left\{ 1, (z,z,1) \right\}$ and $W(H_4)$ is isomorphic to the quotient $G'/Z$. Since the character tables for $W(H_4)$ and its double cover is rather lengthy, we refer readers to \cite[Table II(i), (ii), (iii)]{Re}. We shall denote $\widetilde{\chi}_i$ $(i=35,\ldots, 54)$ for the characters of the $20$ genuine irreducible representations of $\widetilde{W}$ and the order of $\widetilde{\chi}_i$ follows the one in \cite[Table II(ii)]{Re}.

\subsection{The spin modules} \label{ss spin mod}

This subsection describes the spin modules for $C(V_0^{\vee})$, which will be used to study the elliptic representations of $W$ in Section \ref{s ellip}. The detail can be found in, for example, \cite[Section 2.2]{HP}. To construct a spin module, it is easier to start with the complex Clifford algebra $C(V^{\vee})$ of $V^{\vee}$, which is similarly defined as the quotient of the tensor algebra $T(V^{\vee})$ by the ideal generated by elements of the form $\omega \otimes \omega'+\omega'\otimes \omega+2\langle \omega, \omega' \rangle$, for $\omega, \omega' \in V^{\vee}$. The real Clifford algebra $C(V_0^{\vee})$ naturally embeds into $C(V^{\vee})$ as a real subalgebra.

Assume first that $\dim_{\mathbb{C}} V^{\vee}=2m$ is even.  Let $U$ and $U^*$ be $m$-dimensional isotropic subspaces of $V^{\vee}$ with respect to $\langle , \rangle$ so that $V^{\vee}=U \oplus U^*$. 
%Then let $\omega_1, \ldots, \omega_{2m}$ be an orthonormal basis for $V$ with respect to $\langle , \rangle$. Then set
%\[  u_i=\frac{\omega_{2j-1}+\omega_{2j}}{\sqrt{2}},\quad u_i*=\frac{\omega_{2j-1}-\omega_{2j}}{\sqrt{2}} . \]
The simple $C(V^{\vee})$-module $S$ which has $2^m$ dimension is isomorphic to the exterior algebra $\wedge^{\bullet} U$ of $U$ with the action of $C(V^{\vee})$ on $\wedge^{\bullet} U$ determined by the follows:
\begin{align*}
  u \cdot (u_1 \wedge \ldots \wedge u_k) &=u \wedge u_1 \wedge \ldots \wedge u_k, \quad \mbox{ for } u \in U.  \\
    u^* \cdot (u_1 \wedge \ldots \wedge u_k)&=\sum_{i=1}^k (-1)^i 2\langle u^*,u_i\rangle u_1 \wedge \ldots \wedge \hat{u}_i \wedge \dots \wedge u_k,  \quad \mbox{ for } u^* \in U^*.  
\end{align*}
The restriction of $S$ to $C(V_0^{\vee})$ is still an irreducible complex module. The further restriction of $S$ to the even part $C(V_0^{\vee})_{\mathrm{even}}$, however, splits into two inequivalent modules $S^+$ and $S^-$, each of which has dimension $2^{m-1}$ (the choice is arbitrary for the later convenience). Moreover, the restriction of $S$ to the group $\widetilde{W}$ is again irreducible and denoted by $(\gamma, S)$. Similarly, the restriction of $S^+$ and $S^-$ to the group $\widetilde{W}'$ gives rise two inequivalent irreducible $\widetilde{W}'$-representations, denoted again $S^+$ and $S^-$ respectively.

Assume $\dim_{\mathbb{C}} V^{\vee}=2m+1$ is odd. Let $u_1,\ldots, u_{2m+1}$ be an orthonormal basis for $V^{\vee}$. Let $\overline{V}$ be the complex subspace spanned by $u_1,\ldots, u_{2m}$ and let $\overline{U}$ be the complex subspace spanned by $u_{2m+1}$. Then $V^{\vee}=\overline{V}\oplus \overline{U}$. Let $S$ be the spin module for $C(\overline{V})$ as in the case of even dimension. We may define that $u_{2m+1}$ acts on $S$ by $i$ or $-i$ so that $S$ turns into two inequivalent $C(V^{\vee})$ modules of dimension $2^m$, denoted by $S^+$ and $S^-$ (the choice is again arbitrary). These modules remain irreducible when restricted to $C(V_0^{\vee})$. Moreover, the further restriction of $S^+$ and $S^-$ to $\widetilde{W}'$ are still irreducible and are again denoted by $S^+$ and $S^-$.

%%%%%%%%%%%%section of embedding %%%%%%%%%%
%\subsection{Embedding $\mathfrak{so}(V_0^{\vee})$ into $C(V_0^{\vee})$}

%This subsection is based on \cite[Chapter II 6]{BW}. 
%First consider $n=2m$ even. The weights of $S^+$ are exactly of the form
%\begin{equation} \label{eq weights}
% i\frac{1}{2}\left( \sum_{i \in I} \alpha_i - \sum_{i \notin I} \alpha_i \right) ,
%\end{equation}
%where $I$ is any subset of $\left\{1,\ldots, m\right\}$ with $|I|=0$ (mod $2$). The weights of $S^-$ has the same form, but $I$ need to satisfy $|I|=1$ (mod $2$). Thus $S^+$ and $S^-$ are self-dual if $m$ is even, while $S^+$ and $S^-$ are dual to each other if $m$ is odd.

%When $\dim V_0^{\vee}$ is odd, the restriction of $S^+$ and $S^-$ to $C_{\mathrm{even}}(V_0^{\vee})$ are isomorphic modules. The weights of $S^+$ and $S^-$ are again of the form in (\ref{eq weights}), but $I$ can be any subset of $\left\{1,\ldots, m\right\}$ without any further restruction.

%Then, by comparing the weight spaces, as $C(V_0^{\vee})$ modules, we have
%\begin{equation} \label{eq tensor spin}
%S\otimes S = \bigwedge^i V \quad \mbox{ and } S \otimes S=\bigoplus_{i=0}^{[\dim V_0/2]} \bigwedge^{2i} V 
%\end{equation}

The following lemma will be used in Section \ref{s ellip}.
\begin{lemma} \label{lem spin idt}
Let
\[ \wedge^{\pm} V = \sum_i (-1)^i \wedge^i V \]
as a virtual representation of $\widetilde{W}'$. Then in the representation ring of $\widetilde{W}'$,
\[ (S^+-S^-)\otimes (S^+-S^-)^* =\frac{2 }{[W:W']} {\wedge}^{\pm} V ,\]
where $(S^+-S^-)^*$ is the dual of $S^+-S^-$ and $[W:W']$ is the index of $W'$ in $W$.
\end{lemma}

\begin{proof}
See the discussion in \cite[Chapter II 6]{BW}.
%When $\dim V^{\vee}$ is odd, it follows from the equation (\ref{eq tensor spin}) and 
%\[\sgn\oplus \bigoplus_{i=0}^{[\dim V_0/2]} \bigwedge^{2i} V =\bigoplus_{i=0}^{[\dim V_0/2]}\bigwedge^{2i+1}V.\]

\end{proof}

\begin{remark} \label{rmk spin reps}
We shall follow Section \ref{ss dc characters} for the notation of representations.
%for $I_2(n)$ and $H_3$ and \cite[Table II(ii)]{Re} for $H_4$. 
The spin module $S$ in each noncrystallographic case is as follows:
\[ I_2(n)\ (n \mbox{ odd} ): \widetilde{\rho}_{(n-1)/2},\quad I_2(n)\ (n \mbox{ even}): \widetilde{\rho}_{n/2}, \]
\[ H_3: \widetilde{\overline{\chi}}_2^+ \mbox{ or } \widetilde{\overline{\chi}}_2^-,\quad H_4: \widetilde{\chi}_{36}. \]
\end{remark}

%\subsection{The Dirac operator}

%The notion of the Dirac operator for the graded affine Hecke algebra $\mathbb{H}$ was first introduced in \cite{BCT} by D. Barbasch, D. Ciubotaru and P. Trapa. Let $\left\{ \omega_i \right\}$, $\left\{ \omega^i \right\}$ be dual bases of $V_0^{\vee}$ and define
%\[ \widetilde{\omega}_i = \omega_i - \frac{1}{2} \sum_{\beta >0} c_{\beta} (\beta, \omega_i) t_{s_{\beta}}  \]
%so that $\widetilde{\omega}_i^*=-\widetilde{\omega}_i$.
%The Dirac element is defined as 
%\[ \mathcal D = \sum_{i} \widetilde{\omega}_i \otimes \omega^i \in \mathbb{H} \otimes C(V_0^{\vee}). \]
%We fix a spin module $(\phi, S)$ for $C(V_0^{\vee})$ and a fixed $\mathbb{H}$-module $(\pi, X)$. Then, we define the Dirac operator for
%Define the Casimir element in $\mathbb{H}$
%\[ \Omega = \sum_{i=1}^n \omega_i \omega^i \in \mathbb{H} . \]
%Define an element in $C(V_0^{\vee})$:
%\[ \Omega_{\widetilde{W}} =\frac{1}{4} \sum_{\alpha >0 , \beta>0, s_{\alpha}(\beta)<0} c_{\alpha}c_{\beta}\frac{|\alpha^{\vee}|}{|\alpha|}\frac{|\beta^{\vee}|}{|\beta|}f_{\alpha}f_{\beta} . \]

%\[ a(\widetilde{\sigma})=\frac{1}{4}\sum_{\alpha >0} c_{\alpha}^2\langle \alpha, \alpha \rangle - \frac{1}{4}\sum_{\alpha>0, \beta>0, \alpha \neq \beta, s_{\alpha}(\beta)<0} c_{\alpha}c_{\beta} |\alpha||\beta|\frac{\tr_{\widetilde{\sigma}}(f_{\alpha}f_{\beta})}{\tr_{\widetilde{\sigma}}(1)} \]

\section{Distinguished points and solvable points} \label{s def sol}

\subsection{Definitions of distinguished and solvable points} \label{ss def sol}

Recall some terminology from the introduction. Fix a $W$-invariant parameter function $c: R \rightarrow \mathbb{R}$ and write $c_{\alpha}$ for $c(\alpha)$. 

\begin{definition} \cite[Definition 1.4]{HO}
A point $\gamma$ in $V^{\vee}$ is said to be {\it distinguished} if it satisfies
\[   | \left\{ \alpha \in R : (\alpha, \gamma)=c_{\alpha} \right\}| =|\left\{ \alpha \in R : (\alpha, \gamma)=0 \right\}|+|\Delta| .\]
\end{definition}
\noindent
(In fact, any distinguished points are also in $V_0^{\vee}$ and so we may replace $V^{\vee}$ by $V_0^{\vee}$ in the above definition as in the introduction.)

For $J \subset \Delta$, let $V_J$ (resp. $V^{\vee}_J$) be the complex subspace of $V$ (resp. $V^{\vee}$) spanned by the simple roots (resp. coroots) in $J$ and let $R_J=V_J \cap R$ (resp. $R_J^{\vee}=V_J^{\vee} \cap R^{\vee}$). 

\begin{definition}
A point $\gamma$ in $V^{\vee}$ is said to be {\it nilpotent} if it is conjugate to a point $\gamma' \in V_J^{\vee}$ for some $J \subseteq \Delta$ such that $\gamma'$ is a distinguished point in $V_J^{\vee}$ or equivalently, 
\[   | \left\{ \alpha \in R_J : (\alpha, \gamma')=c_{\alpha} \right\}| =|\left\{ \alpha \in R_J : (\alpha, \gamma')=0 \right\}|+|J| .\]
We shall say that $\gamma$ is a nilpotent point associated to $J$ if we want to emphasize the role of $J$.

A point $\gamma$ in $V^{\vee}$ is said to be {\it solvable} if $\gamma$ is nilpotent associated to some subset $J$ of $\Delta$ and $\gamma$ further satisfies 
\[   | \left\{ \alpha \in R : (\alpha, \gamma)=c_{\alpha} \right\}| =|\left\{ \alpha \in R : (\alpha, \gamma)=0 \right\}|+|J| .\]
Let $\mathcal V_{\sol}$ be the set of $W$-orbits of solvable points in $V^{\vee}$.
\end{definition}

The terminology of solvable points is motivated by Theorem \ref{thm def sol} below. We first introduce few notations. Assume $R$ is crystallographic. Set $c_{\alpha}=2$ for all $\alpha \in R$. Let $\mathfrak{g}$ be the semisimple Lie algebra corresponding to $R$. Let $G$ be the connected Lie group of adjoint type over $\mathbb{C}$ with the Lie algebra $\mathfrak{g}$. Fix a Cartan subalgebra $\mathfrak{h} \cong V^{\vee}$. 
For a nilpotent orbit $\mathcal O$ in $\mathfrak{g}$, let $e \in \mathcal O$ and let $\left\{e,h,f \right\}$ be a Jacobson-Morozov standard triple such that 
\[ [h,e]=2e, \quad [h,f]=-2f, \quad [e,f]=h. \]
The semisimple element $h$ depends on the choices involved. However, due to a result of Kostant, the adjoint orbit of $h$ in $\mathfrak{g}$ is independent of the choices. Thus we may associate to each nilpotent orbit $\mathcal O$ a semisimple orbit, denoted by $\mathcal O_h$, via a Jaocbson-Morozov standard triple and indeed this association is an injection (\cite[Theorem 3.5.4]{CM}). Moreover, this association gives rise to the following important map:
\[ \Pi : \mbox{ set of nilpotent orbits in } \mathfrak{g} \rightarrow \mbox{ set of $W$-orbits of nilpotent points in } V^{\vee}, \]
\[ \mathcal O \mapsto  \mathfrak{h} \cap \mathcal O_h  .\]
Note that the intersection $\mathfrak{h} \cap \mathcal O_h$ above is indeed a single $W$-orbit. This follows from the standard fact that two elements in $\mathfrak{h}$ is in the same $W$-orbit if and only if they are in the same adjoint orbit in $\mathfrak{g}$ (\cite[Theorem 2.2.4]{CM}).

From the classification of nilpotent orbits in the Bala Carter theory (see \cite[Lemma 8.2.1, Theorem 8.2.12]{CM}), we have:
\begin{lemma} \label{lem ref bala}
The map $\Pi$ is a well-defined bijection. 
\end{lemma}

The following result says which nilpotent orbits are mapped to the $W$-orbits of solvable points via $\Pi$:

%Recall that the image of the surjective map constructed in \cite{Ci} is the set of nilpotent orbits whose elements have a solvable centralizer in $\mathfrak{g}$. 

%Thus to deal with the non-crystallographic cases where nilpotent orbits are absent, one way is to replace the set $\mathcal N_{\mathrm sol}$ by a suitable set (of $W$-orbits) in $V^{\vee}$ which corresponds nicely as above in the Weyl groups. Our first task is to seek an intrinsic characterization of such set in Proposition \ref{prop eqv def}. 

\begin{theorem} \label{thm def sol}
Let $R$ be a crystallographic root system and set $c \equiv 2$. Let $\mathcal N_{\sol}$ be the set of nilpotent orbits in $\mathfrak{g}$ whose elements have a solvable centralizer. Then $\Pi(\mathcal N_{\sol})=\mathcal V_{\sol}$. In particular, there is a one-to-one correspondence between the set $\mathcal V_{\sol}$ and the set $\mathcal N_{\sol}$. 
%More precisely, the map sending a nilpotent orbit $\mathcal O$ in $\mathcal N_{\sol}$ to a $W$-orbit $\mathfrak{h} \cap \mathcal O_h$ in $\mathcal V_{\sol}$ is a bijection, where $\mathcal O_h$ is defined as above. 
\end{theorem}

A proof of Theorem \ref{thm def sol} is given after Proposition \ref{prop eqv def}.

%We begin with a reformulation of the criterion of a distinguished orbit in the Bala Carter theory. Define a map
%\[ \Pi : \mbox{ set of nilpotent orbits} \rightarrow \mbox{ set of $W$-orbits of nilpotent poitns},\ \mathcal O \mapsto \mathcal O_h \cap \mathfrak{h} .\]

%\begin{lemma}
%The map $\Pi$ is bijective. 
%\end{lemma}

Before stating Proposition \ref{prop eqv def}, we review few facts about Jacobson-Morozov triples (see \cite[Chapter 3]{CM} for the details). Let $\left\{ e,h, f\right\}$ be a Jacobson-Morozov triple. By the representation theory of $\mathfrak{sl}_2$, we have a $\mathbb{Z}$-grading on $\mathfrak{g}$:
\begin{equation} \label{eqn grad sl2}
 \mathfrak{g}=\bigoplus_{i \in \mathbb{Z}} \mathfrak{g}_i , 
\end{equation} 
where $\mathfrak{g}_i=\left\{ Z \in \mathfrak{g} : [h,Z]=iZ \right\}$. Let $\mathfrak{l}=\mathfrak{g}_0$ and let $\mathfrak{n}=\oplus_{i>0} \mathfrak{g}_i$. Then the parabolic subalgebra $\mathfrak{p}=\mathfrak{l}\oplus \mathfrak{n}$ of $\mathfrak{g}$ is called the Jacobson-Morozov parabolic subalgebra corresponding to $\left\{e, h, f\right\}$. Let $\mathfrak{n}_{\mathrm{even}}=\oplus_{i >0} \mathfrak{g}_{2i}$.

%Replacing $\left\{ h,e,f \right\}$ by a conjugate, we may assume $h \in \mathfrak{h} \cong V^{\vee}$. Furthermore, we may even choose $h$ in the fundamental chamber. Hence we can associate an element in $\mathfrak{h}$ to each nilpotent orbit. Indeed, according to a result due to Kostant (see Theorem \cite[Theorem 3.5.4]{CM}, this association is a one-to-one correspondence. 

For a Lie subalgebra $\mathfrak{c} \subset \mathfrak{g}$, write $\mathfrak{g}^{\mathfrak{c}}$ for the centralizer of $\mathfrak{c}$ in $\mathfrak{g}$ i.e. 
\[ \mathfrak{g}^{\mathfrak{c}}=\left\{ X \in \mathfrak{g}: [X,Z]=0 \mbox{ for all } Z\in \mathfrak{c} \right\}. \] For an element $e \in \mathfrak{g}$, write $\mathfrak{g}^e$ for the centralizer of $e$ in $\mathfrak{g}$.

By the representation theory of $\mathfrak{sl}_2$,
$\mathfrak{g}^e$ is the sum of the highest weight spaces of $\mathfrak{sl}_2$-modules. Hence, 
\[ \mathfrak{g}^e = \bigoplus_{i \geq 0} \mathfrak{g}_i^e . \]
Using the representation theory of $\mathfrak{sl}_2$ again, we could see that 
\[ \mathfrak{l}^e=\mathfrak{g}_0^e=\left\{ X \in \mathfrak{g} : [X,Z]=0 \mbox{ for any } Z \in \langle e,h,f \rangle \right\} ,\]
where $\langle e, h, f \rangle$ is the Lie subalgebra of $\mathfrak{g}$ generated by $e, h, f$. 

For a subset $J \subset \Delta$, let $\mathfrak{l}_J=\mathfrak{h} \oplus \sum_{\alpha \in R_J} \mathfrak{g}_{\alpha}$ be the Levi subalgebra associated to $J$, where $\mathfrak{g}_{\alpha}$ is the root space corresponding to $\alpha$. 

\begin{proposition} \label{prop eqv def}
Let $\mathcal O$ be a nilpotent orbit in $\mathfrak{g}$. Let $e$ be an element in $\mathcal O$ such that the minimal Levi subalgebra containing $e$ is equal to $\mathfrak{l}_J$ for some $J \subseteq \Delta$ and a Jacobson-Morozov triple $\left\{ e, h, f \right\}$ is  in $\mathfrak{l}_J$ with $h \in \mathfrak{h}$. Let $\mathfrak{p}=\mathfrak{l}\oplus \mathfrak{n}$ be the Jacobson-Morozov parabolic subalgebra corresponding to $\left\{e,h,f \right\}$.  Then the following conditions are equivalent:
\begin{itemize}
\item[(1)]  $\mathcal O \in \mathcal N_{\mathrm{sol}}$ (notation in Theorem \ref{thm def sol}).
\item[(2)]  $\mathfrak{l}^e = \mathfrak{g}^{\mathfrak{l}_J} $ (and so $\mathfrak{l}^e \subset \mathfrak{h}$).
\item[(3)]  $\dim_{\mathbb{C}} \mathfrak{l}/\mathfrak{g}^{\mathfrak l_J} = \dim_{\mathbb{C}} \mathfrak{n}_{\mathrm{even}}/[\mathfrak{n}_{\mathrm{even}},\mathfrak{n}_{\mathrm{even}}]$.
\item[(4)]  $| \left\{ \alpha \in R : (\alpha, h)=2 \right\}| =|\left\{ \alpha \in R : (\alpha, h)=0 \right\}|+|J|$ .
\end{itemize}
\end{proposition}

\begin{proof}
$(1) \Rightarrow (2)$: By \cite[Lemma 3.7.3]{CM}, $\mathfrak{l}^e$ is reductive. Thus by our assumption (1), $\mathfrak{l}^e$ is contained in some toral subalgebra. For notational simplicity, set $\mathfrak{c}=\mathfrak{l}^{e}$. Then $\mathfrak{g}^{\mathfrak{c}}$ is a Levi subalgebra containing $e$ and so contains some conjugate of $\mathfrak{l}_J$. Then $\mathfrak{c}$ is contained in some conjugate of $\mathfrak{g}^{\mathfrak{l}_J}$. Hence $\dim_{\mathbb{C}} \mathfrak{c} \leq \dim_{\mathbb{C}} \mathfrak{g}^{\mathfrak{l}_J}$. However, we also have $\mathfrak{g}^{\mathfrak{l}_J}=\mathfrak{g}_0^{\mathfrak{l}_J} \subseteq \mathfrak{l}^e=\mathfrak{c}$, where the first equality follows from $h \in \mathfrak{l}_J$, and the second inclusion follows from $e \in \mathfrak{l}_J$. This forces $\mathfrak{c}=\mathfrak{g}^{\mathfrak{l}_J}$. Note that $\mathfrak{g}^{\mathfrak{l}_J} \subseteq \mathfrak{h}$ and so is $\mathfrak{l}^e$.

$(2) \Rightarrow (1)$: By \cite[Lemma 3.7.3]{CM} again, $\mathfrak{g}^e=\mathfrak{l}^{e}\oplus \mathfrak{n}^e$. Then $\mathfrak{g}^{e} \subseteq \mathfrak{h} \oplus \mathfrak{n}$ and so $\mathfrak{g}^e$ is solvable.

$(2) \Rightarrow (3)$:  Recall that $\mathfrak{g}_0=\mathfrak{l}$. Then, by the representation theory of $\mathfrak{sl}_2$, $\ad_e(\mathfrak{l})=\mathfrak{g}_2$. Thus we have $\dim_{\mathbb{C}} \mathfrak{l}-\dim_{\mathbb{C}}\mathfrak{l}^{e}=\dim_{\mathbb{C}} \mathfrak{g}_2$. Since $e \in \mathfrak{n}_{\mathrm{even}}$, $\mathfrak{g}_4=[\mathfrak{n}_{\mathrm{even}},\mathfrak{n}_{\mathrm{even}}]$. Then we also have $\dim_{\mathbb{C}} \mathfrak{g}_2=\dim_{\mathbb{C}} \mathfrak{n}_{\mathrm{even}}/[\mathfrak{n}_{\mathrm{even}},\mathfrak{n}_{\mathrm{even}}]$. Combining equations, we get $\dim_{\mathbb{C}} \mathfrak{l}-\dim_{\mathbb{C}} \mathfrak{l}^{e}=\dim_{\mathbb{C}} \mathfrak{n}_{\mathrm{even}}/[\mathfrak{n}_{\mathrm{even}},\mathfrak{n}_{\mathrm{even}}]$. Now we apply $(2)$ to obtain (3).

$(3) \Rightarrow (2)$: We have seen that $\mathfrak{g}^{\mathfrak{l}_J} \subseteq \mathfrak{l}^{e}$ when proving $(1) \Rightarrow (2)$. Then use the last equality in the argument of $(2) \Rightarrow (3)$ and apply (3) together to obtain (2).

$(3) \Leftrightarrow (4)$ follows from the following three relations: $  |J|= \dim_{\mathbb{C}} \mathfrak{h}-\dim_{\mathbb{C}} \mathfrak{g}^{\mathfrak{l}_J}$,
 \begin{eqnarray*}
 | \left\{ \alpha \in R : (\alpha, h)=2 \right\}| &=&\dim_{\mathbb{C}} \mathfrak{g}_2=\dim_{\mathbb{C}} \mathfrak{n}_{\mathrm{even}}/[\mathfrak{n}_{\mathrm{even}},\mathfrak{n}_{\mathrm{even}}], \\
 |\left\{ \alpha \in R : (\alpha, h)=0 \right\}|&=&\dim_{\mathbb{C}} \mathfrak{l}/\mathfrak{h}=\dim_{\mathbb{C}} \mathfrak{l}-\dim_{\mathbb{C}} \mathfrak{h}.
 \end{eqnarray*}
\end{proof}

\noindent
{\it Proof of Theorem \ref{thm def sol}.}
The statement follows from Lemma \ref{lem ref bala} and the equivalent conditions (1) and (4) in Proposition \ref{prop eqv def}.
\qed

%\begin{remark} \label{rmk terms}
%From the definitions above, 
%\[ \mbox{set of distinguished points } \subset \mbox{set of solvable points} \subset \mbox{ set of nilpotent points } .\]
%On the level of nilpotent orbits, we have the corresponding inclusion:
%\[ \begin{array}{c} \mbox{set of  distinguished} \\ \mbox{  nilpotent orbits } \end{array} \subset \mathcal N_{\mathrm{sol}} \subset \mbox{ set of nilpotent orbits} .\]

%By the Bala-Carter theory on the classification of nilpotent orbits, the nilpotent points in $V^{\vee}$ exactly correspond to the nilpotent orbits in the way explained in the beginning of this subsection. The terminology of distinguished points is introduced in the paper of \cite{HO} by Heckman and Opdam and preciely correspond to the distinguished nilpotent orbits. While solvable points correspond to nilpotent orbits in $\mathcal N_{\mathrm{sol}}$ as seen in Proposition \ref{prop eqv def}, we remark that the set of $\mathcal N_{\mathrm{sol}}$ also contains all the quasidistinguished elements (see \cite{Re} for the definition).

%In the level of graded affine Hecke algebra, those nilpotent points are expected to be the (real) central characters of tempered modules while the distinguished points are expected to support discrete series. The set of tempered modules with solvable points as their central character should contain all elliptic modules (in the sense of \cite{Ree} and \cite{OS}), but in general is a larger set.

%\end{remark}

\subsection{Classification of solvable points}

For the crystallographic cases (with the parameter function $c \equiv 2$), classifying solvable points is equivalent to classifying nilpotent orbits whose centralizer has abelian reductive part (Condition (2) of Proposition \ref{prop eqv def}). The reductive part of the centralizers has been computed for a long time in the literature (see for example \cite{CM} and \cite{Ca}). This goes back to the work of Springer and Steinberg for classical groups, and Alekseev and others for exceptional groups. Then a complete list of nilpotent orbits in $\mathcal N_{\mathrm{sol}}$ can be computed accordingly (also see tables in \cite{Ci}).

We classify all the solvable points for noncrystallographic root systems in the following Proposition \ref{prop class sol}, based on the classification of distinguished points by Heckman and Opdam \cite[Section 4]{HO}.

\begin{proposition} \label{prop class sol}
Let $R$ be a noncrystallographic root system. The set of solvable points in $V^{\vee}$ is described below.
\begin{itemize}
\item[(a)] Suppose $c_{\alpha} \neq 0$ for some $\alpha \in R$. 
\begin{itemize}
\item[(1)] $I_2(n)$ case: Recall $\alpha_1$ and $\alpha_2$ form a basis for $R$. For $i=0, 1 \ldots, n-1$, let $\beta_i \in R_+$ (resp. $\beta_i^{\vee} \in R^{\vee}_+$) be defined by 
\[     \sin \frac{\pi}{n}\beta_i= \sin\frac{(i+1)\pi}{n}\alpha_1+\sin\frac{i\pi}{n} \alpha_2, \]
\[ (\mbox{resp.} \quad \sin \frac{\pi}{n}\beta_i^{\vee}= \sin\frac{(i+1)\pi}{n}\alpha_1^{\vee}+\sin\frac{i\pi}{n} \alpha_2^{\vee}).\]
For $k=1,\ldots, \lfloor   n/2 \rfloor$, let $\left\{ \beta^*_{k-1}, \beta^*_{n-k} \right\}$ be the basis in $V_0^{\vee}$ dual to $\left\{ \beta^{\vee}_{k-1}, \beta^{\vee}_{n-k} \right\}$ (i.e. $\langle \beta^*_{k-1}, \beta^{\vee}_{k-1} \rangle=\langle \beta^*_{n-k}, \beta^{\vee}_{n-k} \rangle=1$ and $\langle \beta^*_{k-1}, \beta^{\vee}_{n-k} \rangle=\langle \beta^*_{n-k}, \beta^{\vee}_{k-1} \rangle=0$). For $k=1,2,\ldots, \lfloor   n/2 \rfloor$, define $\gamma_k=c_{\beta_{k-1}}\beta_{k-1}^*+c_{\beta_{n-k}}\beta_{n-k}^*$.
\begin{itemize}
 \item[(i)] $n$ odd: Further define $\gamma_{(n+1)/2}=\frac{1}{2}\beta_{(n-1)/2}^{\vee}$. Then a point $\gamma$ is solvable if and only if $\gamma$ is conjugate to some $\gamma_k$ $(k=1, \ldots, (n+1)/2)$. Those $\gamma_k$ can be written explicitly as
\[  \gamma_k =\frac{c}{4\sin \frac{2k-1}{2n}\pi \sin \frac{\pi}{2n}} (\alpha_1^{\vee}+\alpha_2^{\vee}) \]
where $c=c_{\alpha}$ for any $\alpha \in R$. Moreover, the point $\gamma_k$ is distinguished if and only if $k \neq (n+1)/2$. 
\item[(ii)] $n$ even:  Then a point $\gamma$ is solvable if and only if $\gamma$ is conjugate to some $\gamma_k$ $(k=1, \ldots, n/2)$. Moreover, the point $\gamma_k$ is distinguished if and only if  
\[ \left(c_{\beta_{k-1}}+c_{\beta_{n-k}}\cos \frac{(2k-1)\pi}{n}\right)\left(c_{\beta_{k-1}}\cos \frac{(2k-1)\pi}{n}+c_{\beta_{n-k}}\right)\neq 0. \]
\end{itemize}
\item[(2)] $H_3, H_4$ case: A point $\gamma$ is solvable if and only if $\gamma$ is distinguished. (See Table \ref{tb H3 quasi} and Table \ref{tb h4 quasi} in Section \ref{s surj} for the list of distinguished points.)
\end{itemize}
\item[(b)] Suppose $c_{\alpha} =0$ for all $\alpha \in R$. A point $\gamma$ is solvable if and only if $\gamma=0$.
\end{itemize} 
\end{proposition}

\begin{proof}
When $c_{\alpha} \equiv 0$, it is straightforward from the definition. We consider the case that $c$ is not identically zero. Since all distinguished points are computed in \cite[Section 4]{HO} by Heckman and Opdam, we could accordingly find out all the nilpotent points. Then straightforward computation could determine which nilpotent points are solvable. We only give the details in $I_2(n)$ with $n$ odd. Recall $\Delta=\left\{ \alpha_1, \alpha_2 \right\}$ is the set of simple roots. Up to conjugation under $W$, there are three possibility for a subset $J$ of $\Delta$: $\emptyset$, $\left\{ \alpha_1 \right\}$ and $\Delta$. When $J=\Delta$, the nilpotent points associated to $\Delta$ are $\gamma_1, \ldots \gamma_{(n-1)/2}$. Those points are distinguished and so solvable. When $J=\left\{ \alpha_1 \right\}$, the only nilpotent point associated to $J$ is $b\alpha_1$, where $b=c_{\alpha_1}/2$. Then it is easy to check $b\alpha_1$ is solvable in $V^{\vee}$ by definitions and is conjugate to $\gamma_{(n+1)/2}$. Finally, for $J=\emptyset$, the only nilpotent point associated to $\emptyset$ is $0$ which is not solvable in $V^{\vee}$. This completes the list of solvable points in $I_2(n)$, $n$ odd.

\end{proof}

\begin{remark}
When $R=I_2(n)$ with $n$ even, the solvable points $\gamma_k$ defined in Proposition \ref{prop class sol} may not be distinct in general even the parameter function $c$ is not identically zero. This also explains why the surjective map $\Phi$ defined later in the proof of Theorem \ref{thm surj} is not a bijection in some situation.
\end{remark}

\section{A correspondence between irreducible genuine $\widetilde{W}$-modules and solvable points} \label{s surj}

In this section, we prove our main result. Recall some notation from the introduction. Let $\mathcal V_{\mathrm{sol}}$ be the set of $W$-orbits of solvable points in $V^{\vee}$. 
Let $\Irr_{\gen}(\widetilde{W})$ be the set of genuine irreducible representations of $\widetilde{W}$. Define an equivalence relation on $\Irr_{\gen}(\widetilde{W})$: $\widetilde{\sigma}_1 \sim \widetilde{\sigma}_2$ if and only if $\widetilde{\sigma}_1=\widetilde{\sigma}_2$ or $\widetilde{\sigma}_1=\sgn \otimes \widetilde{\sigma}_2$, where $\sgn$ is the sign $W$-representation. Fix a symmetric bilinear form $\langle , \rangle$ on $V_0^{\vee}$ as in (\ref{notation inner}). Define an element in the Clifford algebra $C(V_0^{\vee})$:
\[ \Omega_{\widetilde{W}} =- \frac{1}{4} \sum_{\alpha>0, \beta>0,s_{\alpha}(\beta)<0} c_{\alpha}c_{\beta} | \alpha^{\vee}||\beta^{\vee}| f_{\alpha}f_{\beta} ,\]
where $|\alpha^{\vee}|=\langle \alpha^{\vee}, \alpha^{\vee} \rangle^{1/2}$.

\begin{lemma} \label{lem center}
The element $\Omega_{\widetilde{W}}$ is in the center of $C(V_0^{\vee})$.
\end{lemma}

\begin{proof}
Let $\widetilde{R}= \left\{ (\alpha, \beta) \in R_+ \times R_+ : s_{\alpha}(\beta) <0\right\}$. To show $\Omega_{\widetilde{W}}$ is in the center, it suffices to check that for any simple root $\gamma$ and any pair $(\alpha, \beta) \in \widetilde{R}$, $-f_{\gamma}(f_{\alpha}f_{\beta})f_{\gamma}=f_{\alpha'}f_{\beta'}$ for some $(\alpha',\beta') \in \widetilde{R}$. To this end, we divide into four cases. If $s_{\alpha}(\beta) = -\gamma$, then consider $-f_{\gamma}(f_{\alpha}f_{\beta})f_{\gamma}=-f_{\gamma}f_{-s_{\alpha}(\beta)}f_{\alpha}f_{\gamma}=f_{\alpha}f_{\gamma}$. Then $(\alpha, \gamma) \in \widetilde{R}$ as desired. If $\gamma =\alpha$, then $-f_{\alpha}(f_{\alpha}f_{\beta})f_{\alpha}=f_{\alpha}f_{-s_{\alpha}(\beta)}$. Again $(\alpha, -s_{\alpha}(\beta)) \in \widetilde{R}$ as desired. For the case $\alpha =\beta$, we similarly have $-f_{\beta}(f_{\alpha}f_{\beta})f_{\beta}=f_{\alpha}f_{-s_{\alpha}(\beta)}$. The remining case is that $\gamma \neq \alpha$ and $\gamma \neq \beta$ and $s_{\alpha}(\beta) \neq -\gamma$. We have $-f_{\gamma}(f_{\alpha}f_{\beta})f_{\gamma}=(f_{\gamma}f_{\alpha}f_{\gamma})(f_{\gamma}f_{\beta}f_{\gamma})=f_{s_{\gamma}(\alpha)}f_{s_{\gamma}(\beta)}$. Note that $(s_{\gamma}(\alpha),s_{\gamma}(\beta))\in \widetilde{R}$ as desired since $s_{s_{\gamma}(\alpha)}(s_{\gamma}(\beta))=s_{\gamma}(s_{\alpha}(\beta))$. 

\end{proof}

%It is not hard to show that $\Omega_{\widetilde{W}}$ is in the center of $\mathbb{C}[\widetilde{W}]$. 
Now Lemma \ref{lem center} and a version of Schur's lemma imply that $\Omega_{\widetilde{W}}$ acts on an irreducible $\widetilde{W}$-representation $(\widetilde{U}, \widetilde{\chi})$ by a scalar $\widetilde{\chi}(\Omega_{\widetilde{W}})$.

%For a simple $\widetilde{W}$-module $(\widetilde{U},\widetilde{\chi})$, we may naturally regard $\widetilde{U}$ as a $W$-module with the character $\sgn \otimes \widetilde{\chi}$. It is not hard to show that $\Omega_{\widetilde{W}}$ satisfy the relation 
%\[ f_{\alpha} \Omega_{\widetilde{W}}=- \Omega_{\widetilde{W}}f_{\alpha} \]
%for any $f_{\alpha}$ in $\widetilde{W}$. Then $\Omega_{\widetilde{W}}$ sends $(\widetilde{U},\widetilde{\chi})$ to $(\widetilde{U},\sgn \otimes \widetilde{\chi})$ by multiplying a scalar, denoted by $\widetilde{\chi}(\Omega_{\widetilde{W}})$. 

%Thus a version of Schur's lemma implies that $\Omega_{\widetilde{W}}$ acts on a simple $\widetilde{W}$-module $(\widetilde{U}, \widetilde{\chi})$ by a scalar $\widetilde{\chi}(\Omega_{\widetilde{W}})$.

\begin{theorem} \label{thm surj}
Let $(V_0, R, V_0^{\vee}, R^{\vee})$ be a noncrystallographic root system. Fix a symmetric bilinear form $\langle , \rangle$ on $V_0^{\vee}$ as in (\ref{notation inner}). Then there exists a unique surjective map 
\[ \Phi: \Irr_{\gen}(\widetilde{W})/\sim \rightarrow \mathcal V_{\mathrm{sol}} \]
 such that for any $\widetilde{\chi} \in \Irr_{\gen}(\widetilde{W})$ and for any representative $\gamma \in \Phi([\widetilde{\chi}])$,
\begin{eqnarray} \label{eq surj map cond2}
 \widetilde{\chi}(\Omega_{\widetilde{W}})=\langle \gamma,\gamma \rangle .
\end{eqnarray}
Furthermore, $\Phi$ is bijective if and only if $c_{\alpha} \neq 0$ for some $\alpha \in R$, and either:
\begin{enumerate}
\item[(1)] $R=I_2(n)$ ($n$ odd) or $H_3$; or
\item[(2)] $R=I_2(n)$ ($n$ even) with  $\cos(k\pi/n) c'-\cos(l\pi/n) c'' \neq 0$
for any integers $k,l$ with distinct parity and $\cos(k\pi/n), \cos(l\pi/n) \neq \pm 1$, where $c'$ and $c''$ are the two values corresponding to the two distinct $W$-orbits in the parameter function $c$. 
%$\sin (k\pi/n)/\sin (l \pi/n) \neq \pm c_k/c_l$, where $0 \leq k,l \leq n-1$ and $k, l$ have different parity and $k,l \neq n/2$. 
\end{enumerate}\end{theorem}

The proof of Theorem \ref{thm surj} will be done by case-by-case analysis. For $\widetilde{\chi}\in \Irr_{\gen}(\widetilde{W})$, set $a(\widetilde{\chi})=\widetilde{\chi}(\Omega_{\widetilde{W}}).$ 
To begin with, we write down a slightly more explicit formula for $a(\widetilde{\chi})$:
\begin{eqnarray} \label{eqn a form} a(\widetilde{\chi})=\frac{1}{4} \sum_{\alpha>0} c_{\alpha}^2 |\alpha^{\vee}|^2-\frac{1}{4}\sum_{(\alpha,\beta)\in \widetilde{R}}c_{\alpha}c_{\beta}|\alpha^{\vee}||\beta^{\vee}|\frac{\tr_{\widetilde{\chi}}(f_{\alpha}f_{\beta})}{\dim_{\mathbb{C}} \widetilde{\chi} } ,
\end{eqnarray}
where $\tr_{\widetilde{\chi}}(f_{\alpha}f_{\beta})$ is the value of the character of $\widetilde{\chi}$ on $f_{\alpha}f_{\beta}$ and 
\[ \widetilde{R}= \left\{ (\alpha, \beta) \in R_+ \times R_+ : s_{\alpha}(\beta) <0,\ \alpha \neq \beta \right\}. \]

\bigskip

\noindent
{\bf Type $I_2(n)$, $n$ odd} 

Since there is only one $W$-orbit on $R$, we set $c=c_{\alpha}$ for simplicity. Recall from Proposition \ref{prop class sol} that the solvable points are, for $k =1, \ldots, (n+1)/2$,
\[\gamma_k=\frac{c}{4\sin \frac{2k-1}{2n}\pi \sin \frac{\pi}{2n}} (\alpha_1^{\vee}+\alpha_2^{\vee}). \]
The squares of their lengths are 
%(recall that $\langle \alpha_1+\alpha_2, \alpha_1+\alpha_2 \rangle =2$)
 \[ \langle \gamma_k, \gamma_k \rangle = \frac{c^2}{2\sin^2\left(\frac{2k-1}{2n}\pi \right)}  \]
 for $k=1, \ldots, (n+1)/2$.

We next compute the value $a(\widetilde{\chi})$ for $\widetilde{\chi} \in \Irr_{\mathrm{gen}}(\widetilde{W})$.
The following two formulas, whose proofs are elementary (but possibly lengthy), may be useful in computing $a(\widetilde{\chi})$:
\begin{eqnarray} \label{eqn tri1}
\sum_{k=1}^{A} k\cos(rk) &=& -\frac{\sin^2 (r(A+1)/2)}{2\sin^2 (r/2) }+(A+1)\frac{\sin r(1/2+A)}{2\sin (r/2)},
%\frac{-1}{n} \frac{1}{4\sin^2\left(\frac{2r+1}{2n}\pi \right)}+\frac{(-1)^{r+1}}{n}\frac{1}{4\sin\frac{2r+1}{2n}\pi} 
\end{eqnarray}
\begin{eqnarray} \label{eqn tri2}
\sum_{k=1}^{A} \cos (rk) &=& -\frac{1}{2}+\frac{\sin r(1/2+A)}{2\sin (r/2)},
%\frac{-1}{n} \frac{1}{4\sin^2\left(\frac{2r+1}{2n}\pi \right)}+\frac{(-1)^{r+1}}{n}\frac{1}{4\sin\frac{2r+1}{2n}\pi} 
\end{eqnarray}
where $A$ is a positive integer and $r \in \mathbb{R} \setminus \left\{ 2\pi n :n \in \mathbb{Z} \right\}$.

By a simple computation, there are $2n-4k$ pairs of roots $(\alpha, \beta) \in \widetilde{R}$ forming an angle $k\pi/n$ for $k=1,\ldots, (n-1)/2$. Here a pair of roots $\alpha$, $\beta$ forming an angle $\theta$ means that $\cos \theta=\frac{\langle \alpha^{\vee}, \beta^{\vee} \rangle}{|\alpha^{\vee}||\beta^{\vee}|}$.
%, where $|v|$ denotes the length of the vector $v$.

%\begin{table}[h!]
%\begin{tabular}{c c}
%\hline
%$\theta$ & number of pairs \\
%\hline
%$k\pi/n$ & $2n-4k$  \\
%\hline
%\end{tabular}
%\end{table}

Then, by (\ref{eqn a form}) and Table \ref{tb ch odd},
\begin{eqnarray*}
 a(\widetilde{\rho}_i) &=&\frac{1}{2} c^2 n-\frac{1}{2}c^2 \sum_{k=1}^{(n-1)/2} (2n-4k)(-1)^{k+1} \frac{2\cos \left(2ik\pi/n \right)}{2} \\ 
 &=& \frac{1}{2} c^2 n+\frac{1}{2}c^2 \sum_{k=1}^{(n-1)/2} (2n-4k) \cos \left((n-2i)k\pi/n \right) \\
 &=& \frac{c^2}{2\sin^2 \left( \frac{n-2i}{2n}\pi \right)} .
\end{eqnarray*}
The last equality can be deduced from (\ref{eqn tri1}) and (\ref{eqn tri2}) with some further simplification. Similarly,
\[ a(\widetilde{\chi}_1)=a(\widetilde{\chi}_2) =\frac{1}{2} c^2 n -\frac{1}{2}c^2 \sum_{k=1}^{(n-1)/2}(2n-4k)(-1)^{k+1} =\frac{1}{2}c^2 . \]

Note that $[\widetilde{\chi}_1]$ contains $\widetilde{\chi}_1$ and $\widetilde{\chi}_2$ and each $[\widetilde{\rho}_i]$ contains only the character $\widetilde{\rho}_i$. Now, by above computations, the bijection $\Phi$ defined by
\[  [\widetilde{\rho}_i] \stackrel{\Phi}{\mapsto} W\gamma_{(n+1)/2-i} \quad (i=1,\ldots, (n-1)/2) ; \quad   [\widetilde{\chi}_1] \stackrel{\Phi}{\mapsto} W\gamma_{(n+1)/2} ,\]
satisfies the desired property. Here $W \gamma$ means the $W$-orbit of $\gamma$ in $V_0^{\vee}$.

\bigskip

\noindent
{\bf Type $I_2(n)$, $n$ even}

In the case of $n$ even, $\alpha_1$ and $\alpha_2$ are in distinct $W$-orbits and so the unequal parameters case may happen. For notational convenience, set $c_1=c_{\alpha_1}$ and $c_2=c_{\alpha_2}$. For $i=0, 1 \ldots, n-1$, let $\beta_i^{\vee} \in R^{\vee}_+$ defined as in Proposition \ref{prop class sol}. 
%\[     \sin \frac{\pi}{n}\beta_i^{\vee}= \sin\frac{(i+1)\pi}{n}\alpha_1^{\vee}+\sin\frac{i\pi}{n} \alpha_2^{\vee} .\]
Then according to Proposition \ref{prop class sol}, for $k=1, \ldots,  \frac{n}{2} $, the solvable points $\gamma_k$ are determined by 
\begin{equation} \label{def I2even}
  \langle \beta_{n-k}^{\vee},\gamma_k \rangle=c_{\beta_{n-k}}, \quad \langle \beta_{k-1}^{\vee}, \gamma_k \rangle=c_{\beta_{k-1}}.
 \end{equation}
Note that as $n$ is even, $n-k$ and $k-1$ have different parity, and so $\left\{ c_{\beta_{n-k}}, c_{\beta_{k-1}} \right\}=\left\{ c_1, c_2 \right\}$.

Now elementary computation gives
\[ \langle \gamma_k, \gamma_k \rangle = \frac{1}{2\sin^2\left( \frac{n-2k+1}{n}\pi \right)}\left(c_{1}^2-2c_{1}c_{2}\cos\frac{n-2k+1}{n}\pi+c_{2}^2\right). \]

Next step is to compute $a(\widetilde{\chi})$. For $\alpha \in R$, denote the $W$-orbit of $\alpha$ by $W\alpha$. Again we record the number of pairs of roots $(\alpha, \beta) \in \widetilde{R}$ forming certain angles $\theta$: 

\begin{table}[h!]
\caption{The set $\widetilde{R}$ of $I_2(n)$ ($n$ even)} \label{tb I2 even angle}
\begin{tabular}{|l | l|}
\hline
$\theta$ &  number of pairs $(\alpha, \beta)$ in $\widetilde{R}$\\
\hline
$2k\pi/n$                                    &  (1) $n-4k$ for $\alpha, \beta \in W\alpha_1$ \\ 
$k=1,\ldots, \left\lfloor n/4 \right\rfloor$ &   (2) $n-4k$ for $\alpha, \beta \in W\alpha_2$ \\
                                            &   (3) $0$ for $\alpha, \beta$ in distinct $W$-orbits.         \\                              
\hline                                         
$(2k-1)\pi/n$                                   & (1) $2(n-4k+2)$ for $\alpha, \beta$ in distinct $W$-orbits        \\
$k=1, \ldots,  \left\lfloor n/4 \right\rfloor$  %&                 \\
%                                                &  \quad \quad (ii) $1$     if $k = n/4$                 \\
                                                &  (2) $0$ for $\alpha, \beta$ in the same $W$-orbit     \\
\hline
\end{tabular}
\end{table}

Regard $a(\widetilde{\rho_i})$ as a homogeneous polynomial of degree $2$ with indeterminants $c_{1}$ and $c_{2}$. Then, the coefficient of $c_{1}^2$ is equal to
\[ \frac{1}{2}\left(\frac{n}{2}\right)-\frac{1}{2}\sum_{k=1}^{\lfloor n/4 \rfloor} (n-4k)(-1)\frac{ 2\cos \frac{2i-1}{n}(2k)\pi }{2}=\frac{1}{2\sin^2\left(\frac{2i-1}{n}\pi \right)}. \]
The last equality can be again deduced from the formulas (\ref{eqn tri1}) and (\ref{eqn tri2}). The coefficient of $c_{2}^2$ has the same formula as that of $c_{1}^2$. It remains to compute the coefficient of $c_{1}c_2$:
\begin{eqnarray*}
  & & \frac{1}{2}\sum_{k=1, k \mathrm{odd}}^{\lfloor n/2 \rfloor} 2(n-2k)(-1) \frac{2 \cos \frac{2i-1}{n}k\pi}{2} \\
  &=& \sum_{k=1}^{\lfloor n/2 \rfloor} (n-2k)(-1)\frac{2 \cos \frac{2i-1}{n}k\pi}{2}-\sum_{l=1}^{\lfloor n/4 \rfloor} (n-2(2l))(-1)\frac{2\cos\frac{2i-1}{n}(2l)\pi}{2} \\
  &=& \left( \frac{n}{2}-\frac{1}{2\sin^2\left(\frac{2i-1}{2n}\pi\right)}\right)-2\left( \frac{n}{4}-\frac{1}{2\sin^2\left(\frac{2i-1}{n}\pi\right)} \right) \\
  &=& -\frac{\cos\frac{2i-1}{n}\pi}{\sin^2\left( \frac{2i-1}{n}\pi \right)}. \\
\end{eqnarray*}
Hence, we obtain a surjection $\Phi: \Irr_{\mathrm{gen}}(\widetilde{W}) \rightarrow \mathcal V_{\mathrm{sol}}$ defined by
\[   [\rho_i] \stackrel{\Phi}{\mapsto} W\gamma_{n/2-i+1}  \quad (i=1, \ldots, n/2)  .  \]

For the bijectivity of $\Phi$, further calculation shows that $\gamma_k$ in (\ref{def I2even}) are in distinct $W$-orbits if and only if the conditions stated in (2) of the theorem hold. We skip the details of the computation. 

Indeed, if we have $\cos(p\pi/n)c_1-\cos(q\pi/n)c_2=0$ for some $p,q$ with distinct parities and $\cos(p\pi/n),\cos(q\pi/n)\neq \pm 1$, then two distinguished points $\gamma_k$ will be in the same $W$-orbit. This causes the bijectivity fails. Interestingly , if we keep $\cos(p\pi/n)c_1-\cos(q\pi/n)c_2=0$ for some $p,q$ with distinct parities, but we now have $\cos(p\pi/n)=\pm 1$ or $\cos(q\pi/n) \pm 1$, then the bijectivity still holds. However, one of the solvable points $\gamma_k$ will become non-distinguished (Proposition \ref{prop class sol}). 

%Note that referring to the condition (2) in Theorem \ref{thm surj}, when $\cos(k\pi/n) =\pm$ or $\cos(l\pi/n)=\pm$

\bigskip 
\noindent
{\bf  Type $H_3$}

Set $c=c_{\alpha}$ for all $\alpha$. Let $\left\{ \omega_1, \omega_2, \omega_3 \right\}$ be a basis for $V_0^{\vee}$ dual to $\left\{\alpha_1^{\vee}, \alpha_2^{\vee}, \alpha_3^{\vee} \right\}$ so that $\langle \omega_i, \alpha_j^{\vee}\rangle= \delta_{ij}$. By Proposition \ref{prop class sol}, all solvable points are distinguished points. The distinguished points  (from \cite[Table 4.13]{HO}) and their lengths are listed in Table \ref{tb H3 quasi}. We provide some data in Table \ref{tb H3 angle} and Table \ref{tb H3 avalue} so that (\ref{eqn a form}) can be used to compute $a(\widetilde{\chi})$. The explicit form of the bijection $\Phi$ can be read from Table \ref{tb H3 avalue}. Recall $\tau=(1+\sqrt{5})/2$ and $\bar{\tau}=(1-\sqrt{5})/2$.

\begin{table}[!ht]
\caption{ Distinguished/Solvable points in $H_3$}   \label{tb H3 quasi}
\begin{tabular}{l l l}
\hline
Labels & distinguished/solvable points $\gamma$ & $\langle \gamma, \gamma \rangle$ \\
 \hline
$\gamma_1$ &  $c\omega_1+c\omega_2+c\omega_3$ & $(\frac{43}{2}\tau+\frac{19}{2}\bar{\tau})c^2$ \\
$\gamma_2$ &  $(1+\tau)^{-1}(c\omega_1+c\omega_2+c\tau \omega_3)$ & $\frac{11}{2}c^2$ \\
$\gamma_3$ &  $(1+\tau)^{-1}(c\omega_1+c\omega_2+c(1+\tau)\omega_3)$ & $8c^2$    \\
$\gamma_4$ & $(2+3\tau)^{-1}(c(1+\tau)\omega_1+c\tau\omega_2+c\omega_3)$ & $(\frac{19}{2}\tau+\frac{43}{2}\bar{\tau})c^2$ \\
\hline
\end{tabular}
\end{table}

%Again we record the number of pairs of $(\alpha, \beta) \in R_+ \times R_+$ forming various angles $\theta$. Then one may compute the value of $a(\widetilde{\chi})$, which is given in Table \ref{tb H3 avalue}.
\begin{table}[!ht]
\caption{The set $\widetilde{R}$ of $H_3$} \label{tb H3 angle}
\begin{tabular}{|l | c|}
\hline
$\theta$ &  number of pairs $(\alpha, \beta)$ in $\widetilde{R}$ with $\cos\theta=\frac{\langle \alpha, \beta \rangle}{|\alpha||\beta|}$\\
\hline
$\pi/5$                                    &  $36$ \\                             
\hline                                         
$2\pi/5$                                   & $12$   \\
\hline
$\pi/3$                                    &  $20$  \\              
\hline  
\end{tabular}
\end{table}

\begin{table}[!ht] 
\caption{$a(\widetilde{\chi})$ of $H_3$}
%{The values of $a(\wdetilde{\sigma})$ in $H_3$} 
\label{tb H3 avalue}
\begin{tabular}{|c | c| c| c| c| c|} 
\hline
Characters   $\widetilde{\chi}$                     & $\dim_{\mathbb{C}} \widetilde{\chi}$              &   $\tr(-f_{\alpha_1}f_{\alpha_2})$ &  $\tr(-(f_{\alpha_1}f_{\alpha_2})^2)$  & $\tr(-f_{\alpha_2}f_{\alpha_3})$ & $a(\widetilde{\chi} )$\\
\hline
$\widetilde{\chi}_2^+, \widetilde{\chi}_2^-$   &   $2$               &   $-\bar{\tau}$             &   $\tau$              &                    $-1$    &  $\langle \gamma_4, \gamma_4 \rangle$ \\
\hline
$\widetilde{\overline{\chi}}_2^+, \widetilde{\overline{\chi}}_2^-$   &   $2$      &   $-\tau$             &   $\bar{\tau}$              &                   $-1$     &  $\langle \gamma_1, \gamma_1 \rangle$      \\
\hline
$\widetilde{\chi}_6^+$, $\widetilde{\chi}_6^-$  &        $6$               &    $1$     &   $-1$      &  $0$      & $\langle \gamma_2, \gamma_2 \rangle$  \\
     \hline
$\widetilde{\chi}_4^+$, $\widetilde{\chi}_4^-$ &        $4$                 &    $-1$    &    $1$   &   $1$                  &$\langle \gamma_3, \gamma_3 \rangle$ \\
\hline
\end{tabular}
\end{table}

\bigskip
\noindent
{\bf Type $H_4$}

Set $c=c_{\alpha}$ for all $\alpha$. Let $\left\{ \omega_1, \omega_2,\omega_3, \omega_4 \right\}$ be a basis for $V_0^{\vee}$ dual to $\left\{ \alpha_1^{\vee}, \alpha_2^{\vee}, \alpha_3^{\vee}, \alpha_4^{\vee} \right\}$ so that $\langle \omega_i, \alpha_j^{\vee} \rangle=\delta_{ij}$. The distinguished points are shown in Table \ref{tb h4 quasi}. Relevant data in computing $a(\widetilde{\chi})$ are given in Table \ref{tb H4 angle} and Table \ref{tb H4 avalue}.  The characters in Table \ref{tb H4 avalue} are listed in the same order as the ones in \cite[Table II(ii)]{Re}. The explicit surjective map of $\Phi$ can be seen in Table \ref{tb H4 avalue}. 
Recall $\tau=(1+\sqrt{5})/2$ and $\bar{\tau}=(1-\sqrt{5})/2$.

\begin{table}[!ht] 
\caption{ Distinguished/Solvable points in $H_4$} 
\label{tb h4 quasi} 
\begin{threeparttable}
\begin{tabular}{l l  c}
\hline
Labels & distinguished/solvable points $\gamma$ & $\langle \gamma, \gamma\rangle$ \\
 \hline
$\gamma_1$ &  $c\omega_1+c\omega_2+c\omega_3+c\omega_4$ & $(238\tau+94\bar{\tau})c^2$ \\
$\gamma_2$ & $(1+\tau)^{-1}(c\omega_1+c\omega_2+c\tau \omega_3+c\omega_4)$ & $(48\tau+24\bar{\tau})c^2$ \\
$\gamma_3$ & $(1+\tau)^{-1}(c\omega_1+c\omega_2+c\tau\omega_3+c(1+\tau)\omega_4)$ & $(64\tau+28\bar{\tau})c^2$    \\
$\gamma_4$ & $(1+\tau)^{-1}(c\omega_1+c\omega_2+c(1+\tau)\omega_3+c(1+\tau)\omega_4)$ & $(90\tau+42\bar{\tau})c^2$ \\
$\gamma_5$ & $(2+3\tau)^{-1} (c(1+\tau)\omega_1+c\tau\omega_2+c\omega_3+c(1+2\tau)\omega_4)$ & $30c^2$ \\
$\gamma_6$ & $(2+3\tau)^{-1}(c(1+\tau)\omega_1+c\tau\omega_2+c\omega_3+c(1+3\tau)\omega_4)$  & $36c^2$ \\
$\gamma_7$ & $(2+3\tau)^{-1} (c(1+\tau)\omega_1+c\tau\omega_2+c\omega_3+c(2+3\tau)\omega_4)$  & $40c^2$ \\
$\gamma_8$ & $(3+5\tau)^{-1} (c(1+2\tau)\omega_1+c\omega_2+c\tau\omega_3+c\tau\omega_4)$  & $(42\tau+90\bar{\tau})c^2$ \\ %\tnote{1}
$\gamma_9$ & $(2+4\tau)^{-1} (c\omega_1+c\tau\omega_2+c\tau\omega_3+c\omega_4)$ & $17/2c^2$ \\
$\gamma_{10}$ & $(2+3\tau)^{-1} (c\omega_1+c\omega_2+c\tau\omega_3+c\omega_4)$ & $(24\tau+48\bar{\tau})c^2$ \\
$\gamma_{11}$ &$(3+5\tau)^{-1} (c\tau\omega_1+c\tau\omega_2+c\omega_3+c\tau\omega_4)$ & $(28\tau+64\bar{\tau})c^2$ \\
$\gamma_{12}$ & $(5+8\tau)^{-1} (c\omega_1+c(1+2\tau)\omega_2+c\omega_3+c\tau\omega_4)$ & $(94\tau+238\bar{\tau})c^2$ \\
$\gamma_{13}$ &$(1+2\tau)^{-1} (c\omega_2+c\tau\omega_3+c\tau\omega_4)$  &  $16c^2$ \\
$\gamma_{14}$ &$(2+3\tau)^{-1} (c\tau\omega_2+c\tau\omega_3+c\omega_4)$  &  $(18\tau+34\bar{\tau})c^2$ \\
$\gamma_{15}$ &$(1+\tau)^{-1} (c\omega_1+c\omega_2+c(1+\tau)\omega_4)$ & $(34\tau+18\bar{\tau})c^2$ \\
$\gamma_{16}$ &$(1+2\tau)^{-1} (c\omega_2+c\tau\omega_3)$  &  $10c^2$ \\
$\gamma_{17}$ &$(1+\tau)^{-1} c\omega_2$  & $6c^2$ \\
\hline
\end{tabular}
%{\footnotesize
%\begin{tablenotes}
%\item[1] Note that there is a minor typo in \cite[Table 4.14]{HO}.
%\end{tablenotes}
%}
\end{threeparttable}
\end{table}

\begin{table}[!ht]
\caption{The set $\widetilde{R}$ of $H_4$} \label{tb H4 angle}
\begin{tabular}{|l | c|}
\hline
$\theta$ &  number of pairs $(\alpha, \beta)$ in $\widetilde{R}$ with $\cos\theta=\frac{\langle \alpha, \beta \rangle}{|\alpha||\beta|}$\\
\hline
$\pi/5$                                    &  $432$ \\                             
\hline                                         
$2\pi/5$                                   &   $144$   \\
\hline
$\pi/3$                                    &  $400$  \\              
\hline  
\end{tabular}
\end{table}

\begin{table}[!ht]
\caption{ $a(\widetilde{\chi})$ of $H_4$}  \label{tb H4 avalue}
\begin{tabular}{ |c | c| c| c| c| c|} 
\hline
$\mbox{Characters} \widetilde{\chi}$     & $\dim_{\mathbb{C}}\widetilde{\chi}$                &   $\tr(-f_{\alpha_1}f_{\alpha_2})$ &  $\tr(-(f_{\alpha_1}f_{\alpha_2})^2)$  & $\tr(-f_{\alpha_2}f_{\alpha_3})$ & $a(\widetilde{\chi} )$\\
\hline
  $\widetilde{\chi}_{35}$  &   $4$ &    $-2\bar{\tau}$      &   $2\tau$    &    $-2$   &  $\langle \gamma_{12},\gamma_{12} \rangle$ \\
\hline
 $\widetilde{\chi}_{36}$   &   $4$  &    $-2\tau$       &   $2\bar{\tau}$   &    $-2$  &  $\langle \gamma_{1},\gamma_{1} \rangle$   \\
\hline
 $\widetilde{\chi}_{37}$   &   $12$ &    $2$      &     $-2$       &      $0$    &     $\langle \gamma_{17},\gamma_{17} \rangle$ \\
 \hline
  $\widetilde{\chi}_{38}$   &   $8$ &   $-2$     &  $2$       &     $2$       &   $\langle \gamma_{13},\gamma_{13} \rangle$ \\
 \hline
  $\widetilde{\chi}_{39}$   &   $16$ &   $2\bar{\tau}$    &   $-2\tau$    &     $-2$   &  $\langle \gamma_{3},\gamma_{3} \rangle$ \\
 \hline
  $\widetilde{\chi}_{40}$   &   $16$ &   $2\tau$  &  $-2\bar{\tau}$  &    $-2$   &  $\langle \gamma_{11},\gamma_{11} \rangle$ \\
 \hline
  $\widetilde{\chi}_{41}$   &   $48$ &   $-2$   &    $2$         &   $0$       &  $\langle \gamma_{6},\gamma_{6} \rangle$ \\
 \hline
  $\widetilde{\chi}_{42}$   &   $32$ &   $2$   &     $-2$          &   $2$  &   $\langle \gamma_{9},\gamma_{9} \rangle$ \\
 \hline
  $\widetilde{\chi}_{43}$   &   $12$ &   $2$   &    $-2$          &   $0$    &  $\langle \gamma_{17},\gamma_{17} \rangle$ \\
 \hline
  $\widetilde{\chi}_{44}$   &   $12$ &   $2$   &   $-2$            &  $0$     &  $\langle \gamma_{17},\gamma_{17} \rangle$      \\
 \hline
  $\widetilde{\chi}_{45}$   &   $12$ &   $-2\tau^2$    &   $2\bar{\tau}^2$   & $0$   &  $\langle \gamma_{4},\gamma_{4} \rangle$   \\
 \hline
  $\widetilde{\chi}_{46}$   &   $12$ &   $-2\bar{\tau}^2$ &    $2\tau^2$  & $0$      &  $\langle \gamma_{8},\gamma_{8} \rangle$  \\
 \hline
  $\widetilde{\chi}_{47}$   &   $36$ &    $2\bar{\tau}$   &     $-2\tau$   & $0$     &   $\langle \gamma_{15},\gamma_{15} \rangle$  \\
 \hline
  $\widetilde{\chi}_{48}$   &   $36$ &    $2\tau$    &   $-2\bar{\tau}$     &   $0$  &  $\langle \gamma_{14},\gamma_{14} \rangle$   \\
 \hline
  $\widetilde{\chi}_{49}$   &   $24$ &    $-2\tau$  &     $2\bar{\tau}$    &     $0$  &  $\langle \gamma_{2},\gamma_{2} \rangle$   \\
 \hline
  $\widetilde{\chi}_{50}$   &   $24$ &    $-2\bar{\tau}$    &     $2\tau$   &    $0$  &  $\langle \gamma_{10},\gamma_{10} \rangle$\\
 \hline
  $\widetilde{\chi}_{51}$   &   $20$ &    $0$      &  $0$             &       $2$  & $\langle \gamma_{16},\gamma_{16} \rangle$     \\
 \hline
  $\widetilde{\chi}_{52}$   &   $20$ &    $0$     &  $0$       &       $2$    &    $\langle \gamma_{16},\gamma_{16} \rangle$  \\
 \hline
  $\widetilde{\chi}_{53}$   &   $60$ &     $0$     &  $0$      &       $0$      &  $\langle \gamma_{5},\gamma_{5} \rangle$  \\
 \hline
  $\widetilde{\chi}_{54}$   &   $40$ &     $0$      &  $0$       &       $-2$      & $\langle \gamma_{7},\gamma_{7} \rangle$   \\
 \hline
 \end{tabular}
\end{table}

\begin{remark}
In the crystallographic cases, Ciubotaru \cite{Ci} constructed a surjective map $\Theta$ from $\Irr_{\gen}(\widetilde{W})/\sim$ to $\mathcal N_{\sol}$. Composing the map $\Theta$ and the bijection $\Pi: \mathcal N_{\sol} \rightarrow \mathcal V_{\sol}$ defined in Section \ref{ss def sol}, we would again obtain a surjective map as the one in Theorem \ref{thm surj}.
\end{remark}

%%%%%%%%%%Compatability of the known results%%%%%%

%\begin{remark}
%We may try to see how Theorem \ref{thm surj} works with the content of tempered $\mathbb{H}$-modules in the case of $I_2(n)$. When $n$ is odd, 

%In the case that $W=I_2(n)$ ($n$ even), $\gamma_k$'s defined in (\ref{def I2even}) are not necessarily distinct. All of them are different only under the condition that for any $1 \leq k \neq l \leq [(n-1)/2)]$,
%\[  \frac{\sin(k\pi/n) }{\sin (l\pi/n)} \neq \frac{c_k}{c_l}.\]

%\end{remark}

%%%%%%%%%%remark for covering some cases in the Weyl group
%\begin{remark}
%The proof in Theorem \ref{thm surj} also covers the crystallographic root systems $I_2(3)$, $I_2(4)$ and $I_2(6$). Now fix the inner product $\langle ,\rangle_{B_2}$ to be the usual one in $B_2$ (i.e. $\langle \alpha_1,\alpha_1 \rangle_{B_2}=2$ and $\langle \alpha_2,\alpha_2 \rangle_{B_2}=1$). Choosing parameter functions $c'$ and $c$  such that $c'_{\alpha_1}=\sqrt{2}c'_{\alpha_2}=\sqrt{2}$ and $c_{\alpha_1}=c_{\alpha_2}=1$, we obtain isomorphic Hecke algebras $\mathbb{H}(I_2(2), c')$ and $\mathbb{H}(B_2,c)$. Then our result for $\mathbb{H}(I_2(2),c')$ in Theorem \ref{thm surj} is analogue to the result for $\mathbb{H}(B_2,c)$ in \cite{Ci}. The cases for $I_2(3)$ and $I_2(6)$ are also similar.

%\end{remark}

\section{Elliptic representations of $W$} \label{s ellip}

\subsection{Elliptic representation theory of finite groups} \label{ss ellip finite}

This section is based on \cite[Section 3.1]{Ree}. Let $\Gamma$ be a finite group and let $V$ be a $\Gamma$-representation over $\mathbb{C}$. Then an element $\gamma \in \Gamma$ is said to be {\it elliptic} if $\gamma$ does not have any fixed point in $V$, or equivalently, $\det_V(\gamma-1)\neq 0$, where $\det_V$ is the determinant function for the linear operators on $V$. A conjugacy class $C$ in $\Gamma$ is called {\it elliptic} if $C$ contains an elliptic element. 

Let $R(\Gamma)$ be the representation ring of $\Gamma$ over $\mathbb C$. Let $\mathcal L$ be a collection of subgroups $L$ of $\Gamma$ such that $V^{L} \neq 0$. Here $V^L=\left\{ v \in V: \gamma \cdot v=v \mbox{ for all } \gamma \in L \right\}$ is the fixed point set of $L$ in $V$. For every $L \in \mathcal L$, let $\Ind_L^{\Gamma}: R(L) \rightarrow R(\Gamma)$ be the induction map and define
\begin{align} \label{eqn Rind}
R_{\ind}(\Gamma)=\sum_{L \in \mathcal L} \Ind_{L}^{\Gamma} (R(L)) \subset R(\Gamma)  , \\
\label{eqn Rellip} \overline{R}(\Gamma)=R(\Gamma)/R_{\ind}(\Gamma) .
\end{align}
The quotient $\overline{R}(\Gamma)$ is called the {\it elliptic representation ring} of $\Gamma$.

Define a bilinear form $\langle . ,. \rangle_{\Gamma}$ on $R(\Gamma)$:
     \[ \langle \sigma_1,\sigma_2 \rangle_{\Gamma} =\frac{1}{|\Gamma|} \sum_{\gamma \in \Gamma} \tr_{\sigma_1}(\gamma) \overline{\tr_{\sigma_2}(\gamma)}= \dim_{\mathbb{C}} \Hom_{\Gamma}(\sigma_1, \sigma_2) , \]
where $\tr_{\sigma_i}(\gamma)$ is the value of the character of $\sigma_i$ on $\gamma$.

Define an elliptic pairing $e_{\Gamma}$ on $R(\Gamma)$:
\begin{equation*}
 e_{\Gamma}(\sigma_1, \sigma_2) = \sum_{i=0}^{\dim_{\mathbb{C}} V} (-1)^i\dim_{\mathbb{C}} \Hom_{\Gamma}(\sigma_1 \otimes \wedge^iV, \sigma_2)  =  \langle \sigma_1 \otimes \wedge^{\pm}V, \sigma_2 \rangle_{\Gamma}, 
 \end{equation*}
where $\wedge^{\pm} V=\sum_{i} (-1)^i \wedge^i V$.

 It is shown in \cite{Ree} that the radical $\mathrm{rad}(e_{\Gamma})$ of $e_{\Gamma}$ is precisely $R_{\ind}(\Gamma)$. Thus the elliptic pairing $e_{\Gamma}$ induces a non-degenerate bilinear form on $\overline{R}(\Gamma)$, which is still denoted $e_{\Gamma}$, and furthermore 
\begin{equation} \label{eq dim ellip}
\dim_{\mathbb{C}} \overline{R}(\Gamma)=|\mathcal C_{\ellip}(\Gamma)|, 
\end{equation}
where $\mathcal C_{\ellip}(\Gamma)$ is the set of elliptic conjugacy classes in $\Gamma$.

\subsection{Elliptic representations of real reflection groups}
We now specialize to $\Gamma=W$ a finite real reflection group and $V$ the reflection representation associated to $W$. We give two general properties about the elliptic representations of $W$.

A subgroup $L$ of $W$ is a proper parabolic subgroup if $L=\langle s_{\alpha_i} : \alpha_i \in J \rangle$ for some $J \subsetneq \Delta$. Let $\mathcal W$ be the set of proper parabolic subgroups of $W$.

\begin{proposition} \label{prop fix gp}
For any $L \in \mathcal L$ (notation in Section \ref{ss ellip finite}), there exists a proper parabolic subgroup $P$ of $W$ and $w \in W$ such that $L \subseteq wPw^{-1}$. In particular,
 \begin{equation} \label{eqn fix gp}
 R_{\ind}(W) = \sum_{P \in \mathcal W} \Ind_P^W(R(P)) . 
\end{equation}
\end{proposition}

\begin{proof}
Let $L \in \mathcal L$ and let $v \in V^L$. There exists an element $w \in W$ such that $w\cdot v$ is in the fundamental chamber of $V^{\vee}$ (i.e. $\mathrm{Re}(\alpha_i,v)\geq 0$ for all simple roots $\alpha_i$). Then the fixed point subgroup of $w\cdot v$ is the proper parabolic subgroup $P$ generated by simple reflections fixing $w\cdot v$. Hence $L \subset w^{-1}Pw$ as desired. Since $\Ind_L^W(R(L)) \subseteq \Ind_{w^{-1}Pw}^W(R(w^{-1}Pw))=\Ind_P^W(R(P))$, we have the second assertion.

%Let $L \in \mathcal L$. Find an orthonormal basis $\left\{ v_1, \ldots, v_l \right\}$ for $V^L$. We will proceed by induction on $l$. When $l=1$, by applying some $w \in W$ on the basis, we may assume $v_1$ in the fundamental chamber (i.e. $(v_1, \alpha_i)\geq 0$ for all $\alpha_i$). Then the fixed point subgroup of $v_1$ is exactly the parabolic subgroup generated by simple reflections fixing $v_1$.

%Now we consider an arbitrary $l$. Again, we may assume $v_1$ is in the fundamental chamber. Let $L'$ be the parabolic subgroup fixing $v_1$. Then $L$ is a subgroup of $L'$. We decompose each $v_i$ into $a(v_i)+b(v_i)$ such that $a(v_i) \in V_L$ and $b(v_i) \in V_L^{\bot}$, where $V_L^{\bot}$ is the orthogonal complement of $V_L$ in $V$. By induction hypothesis, the group fixing those $a(v_1), a(v_2), \ldots, a(v_r)$ is some parabolic subgroup $P \subset P' \subset W$. Then we see $L=P$ as desired.

\end{proof}

Recall that $W'$ is defined in Section \ref{ss clifford}.
\begin{proposition} \label{prop ell det}
Let $w$ be an elliptic element in $W$. Then $\det_{V}(w)=(-1)^{r}$, where $r=\dim_{\mathbb C} V$. In particular, the set of elliptic elements is a subset of $W'$. 
%Consequently, $w$ is not elliptic, if 
%\begin{itemize}
%\item[(1)] $\dim V_0$ is even and $\det_{V_0}(w)=-1$, or
%\item[(2)] $\dim V_0$ is odd and $\det_{V_0}(w)=1$.
%\end{itemize}
%\end{lemma}
\end{proposition}

\begin{proof}
Realize $w$ as a matrix in $O(V)$. Since $w$ is a real matrix, $\zeta$ is an eigenvalue of $w$ if and only if the complex conjugate of $\zeta$ is also an eigenvalue of $w$. Hence, the number of real eigenvalues of $w$ has the same parity as $r$. Since $w$ is elliptic, the only possible real eigenvalue of $w$ is $-1$. Hence $\det(w)=(-1)^r$.

%Let $p(x)=\det(xI-w)$. 

%The ellipticity of $w$ implies that $x-1$ is not a factor of $p(x)$. It is well-known that any (irreducible) cyclotomic polynomial except $x-1$ has constant term $1$. Since $p(x)$ is the product of cyclotomic polynomials other than $x-1$, the constant term of $p(x)$ is $1$. Thus the product of eigenvalues of $w$ is $(-1)^n$.

\end{proof}

\begin{remark} \label{rmk ellip no}
We record the number of elliptic conjugacy classes in each noncrystallographic case:
\[ I_2(n), n \mbox{ odd}: (n-1)/2, \quad I_2(n), n \mbox{ even}: n/2, \quad H_3: 4, \quad H_4: 20 .\]
This also gives the dimension of $\overline{R}(W)$ by (\ref{eq dim ellip}).

\end{remark}

\subsection{An isometry between $\overline{R}(W)$ and $R_{\gen}(\widetilde{W}')$}
Let $ R_{\gen}(\widetilde{W})$ be the subspace of $ R(\widetilde{W})$ spanned by genuine irreducible representations of $\widetilde{W}$.

Let $\mathcal C^0$ be the set of conjugacy classes $C$ in $W$ such that $p^{-1}(C)$ splits into two distinct conjugacy classes in $\widetilde{W}$. Then 
\begin{equation}
      \dim_{\mathbb{C}} R_{\gen}(\widetilde{W})= |\mathcal C(\widetilde{W})|-|\mathcal C(W)|=|\mathcal C^0|,
\end{equation}
where $\mathcal C(W)$ (resp. $\mathcal C(\widetilde{W}))$ is the set of conjugacy classes in $W$ (resp. in $\widetilde{W}$).
\begin{remark}
The cardinality $|\mathcal C^0|$ in each case is as follows:
\[  I_2(n), n \mbox{ odd}: (n+3)/2, \quad I_2(n), n \mbox{ even}: n/2 \quad H_3: 8, \quad H_4: 20. \]
\end{remark}

Recall that the spin representation $S$ of $W$ is defined in Section \ref{ss spin mod}. Define a map
\[ i_S : R(W) \rightarrow R_{\gen}(\widetilde{W}),\ \sigma \mapsto \sigma \otimes S .\]

\begin{proposition}
Let $R$ be a noncrystallographic root system. The map $i_S$ is surjective if and only if $R$ is $I_2(n)$ with $n$ even, $H_3$ or $H_4$. 
\end{proposition}
\begin{proof}
Since the genuine $\widetilde{W}$ representations are determined by the values of their characters on the conjugacy classes in $p^{-1}(\mathcal C_0)$, it suffices to check that the character of $S$ does not vanish on all conjugacy classes in $p^{-1}(\mathcal C^0)$ if and only if $R$ is not $I_2(n)$ with $n$ odd. This can be verified by using Remark \ref{rmk spin reps} with Table \ref{tb ch odd}, Table \ref{tb ch even} for $I_2(n)$ and Table \ref{tb chg h3} for $H_3$ and \cite[Table II(b)]{Re} for $H_4$.

\end{proof}

The above proposition for the crystallographic root systems is shown in \cite[Lemma 3.3.3]{Ci}.

We define some new notation. Let $R_{\gen}(\widetilde{W}')$ be the subspace of $R(\widetilde{W}')$ generated by genuine representations of $\widetilde{W}'$. Define another map
\[ i_S^{\pm} :R(W) \rightarrow R_{\gen}(\widetilde{W}'),\ \sigma \mapsto \sigma \otimes (S^+-S^-). \]
It has been discovered in \cite{CT} that the map $i^{\pm}_S$ is more interesting to be studied and is related to the Dirac index of modules of graded affine Hecke algebras.

%Define $i_S^{\pm}: R(W) \rightarrow R_{\gen}(\widetilde{W}')$,
%\[   \sigma \mapsto \sigma \otimes (S^+-S^-) . \]

\begin{proposition} \label{prop isometry}
The map $i^{\pm}_S$ satisfies
\[   2e_W(\sigma_1,\sigma_2)= \langle i^{\pm}_S(\sigma_1) , i^{\pm}_S(\sigma_2) \rangle_{\widetilde{W}'}, \]
for $\sigma_1,\sigma_2 \in R(W)$.
\end{proposition}

\begin{proof}
This is \cite[Proposition 3.1]{CT}. Indeed, it follows from the computation below:
\begin{eqnarray*}
        \langle i^{\pm}(\sigma_1), i^{\pm}(\sigma_2) \rangle_{\widetilde{W}'}                    &=&  \langle \sigma_1 \otimes (S^+-S^-)\otimes (S^+-S^-)^*, \sigma_2 \rangle_{\widetilde{W}'} \\
                   &=& \frac{2}{[W:W']} \langle \sigma_1 \otimes  \wedge^{\pm}V, \sigma_2 \rangle_{\widetilde{W}'} \\
                   &=& \frac{2}{[W:W']} \langle \sigma_1 \otimes  \wedge^{\pm}V, \sigma_2 \rangle_{W'} \\
                   &=& \frac{2}{|W|} \sum_{w \in W'} \tr_{\sigma_1}(w) \tr_{\wedge^{\pm}V}(w) \overline{\tr_{\sigma_2}(w)} \\
                   &=& \frac{2}{|W|} \sum_{w \in W} \tr_{\sigma_1}(w) \tr_{\wedge^{\pm}V}(w) \overline{\tr_{\sigma_2}(w)} .\\
\end{eqnarray*}
In the first equality, $(S^+-S^-)^*$ is the dual of $S^+-S^-$. The second equality follows from Lemma \ref{lem spin idt}. The remaining nontrivial equality is the fifth one. By \cite[Lemma 2.1.1]{Ree}, $\tr_{\wedge^{\pm}V}(w)=\det_V(1-w)$ for any $w \in W$ and so $\tr_{\wedge^{\pm}V}$ vanishes off the set of elliptic conjugacy classes in $W$. Now the fifth equality follows from Proposition \ref{prop ell det}. Thus we have 
\[             \langle i^{\pm}(\sigma_1), i^{\pm}(\sigma_2) \rangle_{\widetilde{W}'}       = 2\langle \sigma_1 \otimes \wedge^{\pm} V, \sigma_2 \rangle_{W} = 2e_W(\sigma_1, \sigma_2) .\]

%The forth equality follows from the fact that the character of $\otimes \wedge^{\pm} V$ takes $0$ on any element in the conjugacy classes which do not belong to $\mathcal C^0$.

\end{proof}

In general, the map $i^{\pm}_S$ is not an injection since the elliptic pairing $e_W$ may be degenerate. To remedy this, we consider the elliptic representation ring $\overline{R}(W)= R(W)/\mathrm{rad}(e_W)$. Then the map $i^{\pm}_S$ descends to an injection from $\overline{R}(W)$ to $R_{\gen}(\widetilde{W}')$ by Proposition \ref{prop isometry}. Then we also have \begin{eqnarray} \label{eqn equality im ell}
\dim_{\mathbb{C}} \im(i_S^{\pm})&=&|\mathcal C_{\ellip}(W)|. 
\end{eqnarray}

We now define an involution $\iota$ on $R_{\gen}(\widetilde{W}')$. When $\dim_{\mathbb{C}} V$ is odd, define  $\iota(\widetilde{\sigma})=\sgn \otimes \widetilde{\sigma}$. When $\dim_{\mathbb{C}} V$ is even, there is an outer automorphism of order $2$ on $\widetilde{W}'$ by the conjugation of an element in $\widetilde{W} -\widetilde{W}_{\mathrm{even}}$. Then $\iota$ is defined to be the involution on $R_{\gen}(\widetilde{W}')$ induced from the outer automorphism. In particular, we have $\iota(S^{\pm})=S^{\mp}$, and
\[   \iota(\sigma \otimes (S^+-S^-))=-\sigma \otimes (S^+-S^-) \mbox{ for $\sigma \in R(W)$ }.\]
Thus $\im(i_{S}^{\pm}) \subseteq \ker(\iota+1)$. To study when another inclusion $\im(i_{S}^{\pm}) \supseteq \ker(\iota+1)$ holds, we need a simple lemma.

Let $\widetilde{\sigma} \in \widetilde{W}$. First consider $\dim_{\mathbb{C}} V$ is even. We divide into two cases. If the restriction of $\widetilde{\sigma}$ to $\widetilde{W}'$ splits into two representations, then let $\widetilde{\sigma}^+$ and $\widetilde{\sigma}^-$ be those two representations (the choice is arbitrary). Otherwise the restriction of $\widetilde{\sigma}$ to $\widetilde{W}'$ is irreducible and let $\widetilde{\sigma}^+=\widetilde{\sigma}^-=\widetilde{\sigma}$. For $\dim_{\mathbb{C}} V$ odd, let $\widetilde{\sigma}^+=\widetilde{\sigma}$ and $\widetilde{\sigma}^-=\sgn \otimes \widetilde{\sigma}$. By definition, we have $\iota(\widetilde{\sigma}^{\pm})=\sigma^{\mp}$ no matter $\dim_{\mathbb{C}} V$ is odd or even.

%A representation $\widetilde{\sigma} \in \Irr_{\gen}(\widetilde{W})$ is said to be self-associate if $\widetilde{\sigma}=\sgn \otimes \widetilde{\sigma}$. 

\begin{lemma} \label{lem equality ik}
Let $\widetilde{\sigma} \in \Irr_{\gen}(\widetilde{W})$. 
\begin{itemize}
\item[(1)] Assume $\dim_{\mathbb{C}} V$ is even. Then $\widetilde{\sigma}=\sgn \otimes \widetilde{\sigma}$ if and only if $\widetilde{\sigma}^+\neq \widetilde{\sigma}^-$. Moreover,
\[ \dim_{\mathbb{C}}\ker(\iota+1)=|\left\{ \widetilde{\sigma} \in \Irr_{\gen}(\widetilde{W}): \widetilde{\sigma}= \sgn \otimes \widetilde{\sigma} \right\}|.\]
%Moreover, the number of $\widetilde{\sigma} \in \Irr(\widetilde{W})$ with $\widetilde{\sigma}^+=\widetilde{\sigma}^-$ is equal to $|\mathcal C^0\cap \mathcal C_{\mathrm{odd}}|$.
\item[(2)] Assume $\dim_{\mathbb{C}} V$ is odd. Then $\widetilde{\sigma}=\sgn \otimes \widetilde{\sigma}$ if and only if $\widetilde{\sigma}^+=\widetilde{\sigma}^-$. Moreover,
\[ 2\dim_{\mathbb{C}}\ker(\iota+1)=|\left\{ \widetilde{\sigma} \in \Irr_{\gen}(\widetilde{W}):   \widetilde{\sigma}\neq\sgn \otimes \widetilde{\sigma} \right\}| . \]
%Moreover, the number of $\widetilde{\sigma} \in \Irr(\widetilde{W})$ with $\widetilde{\sigma}^+=\widetilde{\sigma}^-$ is equal to $\frac{1}{2}|\mathcal C^0 \cap \mathcal C_{\mathrm{even}}|$.
\end{itemize}
\end{lemma}

\begin{proof}
We only consider $\dim_{\mathbb{C}} V$ is even. The case for $\dim_{\mathbb{C}} V$ odd is similar (and easier). Note that
\begin{eqnarray} \label{eqn even split}
     \langle \widetilde{\sigma}, \widetilde{\sigma} \rangle_{\widetilde{W}} &=& %\frac{1}{|\widetilde{W}|}\sum_{w \in \widetilde{W}} |\widetilde{\sigma}(w)|^2 \\
%                                                  &=& \frac{1}{2|\widetilde{W}'|}\sum_{w \in \widetilde{W}'}|\widetilde{\sigma}(w)|^2+\frac{1}{|\widetilde{W}|}\sum_{w \in \widetilde{W}_{\mathrm{odd}}}|\widetilde{\sigma}(w)|^2 \\
                                              \frac{1}{2}\langle \widetilde{\sigma}|_{\widetilde{W}'}, \widetilde{\sigma}|_{\widetilde{W}'} \rangle_{\widetilde{W}'}+\frac{1}{|\widetilde{W}|}\sum_{w \in \widetilde{W}_{\mathrm{odd}}}|\tr_{\widetilde{\sigma}}(w)|^2, 
  \end{eqnarray}
where $\widetilde{W}_{\mathrm{odd}}=p^{-1} ( \det_{V_0^{\vee}}^{-1}(-1) \cap W)$. Then $\langle \widetilde{\sigma}|_{\widetilde{W}'}, \widetilde{\sigma}|_{\widetilde{W}'} \rangle_{\widetilde{W}'}=1$ or $2$. Note $\langle \widetilde{\sigma}|_{\widetilde{W}'}, \widetilde{\sigma}|_{\widetilde{W}'} \rangle_{\widetilde{W}'}=2$ if and only if $|\tr_{\widetilde{\sigma}}(w)|=0$ for all $w \in \widetilde{W}_{\mathrm{odd}}$. This proves the first assertion of (1). Let 
\[B=\left\{ \widetilde{\sigma}^+-\widetilde{\sigma}^- : \widetilde{\sigma} \in \Irr_{\gen}(\widetilde{W}) \mbox{ and } \widetilde{\sigma}^+ \neq \widetilde{\sigma}^- \right\}.\] The second assertion will follow from the first assertion if we show $B$ forms a basis for $\ker(\iota+1)$. To this end, note that $B$ spans $\ker(\iota+1)$ since $\iota(\widetilde{\sigma}^+)=\widetilde{\sigma}^-$, and $B$ is linearly independent by considering the inner product $\langle , \rangle_{\widetilde{W}'}$.

\end{proof}

\begin{proposition}\label{prop equality ik}
Let $R$ be a noncrystallographic root system. Then $\im(i^{\pm}_S)=\ker(1+\iota)$.
\end{proposition}
\begin{proof}
Since $\im(i^{\pm}_S) \subset \ker(1+\iota)$, it suffices to show that $\dim_{\mathbb{C}} \im(i^{\pm}_S)=\dim_{\mathbb{C}} \ker(1+\iota)$. We record the number $|\left\{ \widetilde{\sigma} \in \Irr_{\gen}(\widetilde{W}): \widetilde{\sigma}= \sgn \otimes \widetilde{\sigma} \right\}|$ for $\dim_{\mathbb{C}} V$ even:
\[I_2(n) \mbox{ $n$ odd}: (n-1)/2,\quad I_2(n) \mbox{ $n$ even}: n/2,\quad H_4: 20 . \] 
The number $|\left\{ \widetilde{\sigma} \in \Irr_{\gen}(\widetilde{W}):   \widetilde{\sigma}\neq\sgn \otimes \widetilde{\sigma} \right\}|$ for $H_3$ is  $8$. Then the result follows from (\ref{eqn equality im ell}), Remark \ref{rmk ellip no} and Lemma \ref{lem equality ik}.
\end{proof}

We remark that the above proposition is not true for some crystallographic cases, for example $A_n$ ($n \geq 5$). 
%\begin{proposition}
%Let $R$ be a root system. Then $\im(i^{\pm}_S)=\ker(1+\iota)$ if and only if $W$ is of type $A_2$, $A_3$, $B_n/C_n$, $D_n$ ($n$ odd), $E_8$, $F_4$, $G_2$, $I_2(n)$, $H_3$ and $H_4$.
%\end{proposition}

%\begin{proof}
%Since $\im(i^{\pm}_S) \subset \ker(1+\iota)$, it suffices to check when  $\dim \im(i^{\pm}_S)$ and $\dim \ker(1+\iota)$. We record the number of self associate representations in $\Irr_{\gen}(\widetilde{W})$ for $n=\dim V_0$ even:
%$A_{n-1}$: the number of partitions of $n$ into distinct parts with even number of even parts, \\
%$B_n/ C_n$: the number of partitions of $n$, \\
%$D_n$: twice of the number of even parts minus the number of odd parts \\
%$E_6: 5,\quad E_8: 30, \quad F_4: 9, \quad G_2: 3$ \\
%$ I_2(l) \mbox{ $l$ odd}: (n-1)/2,\quad I_2(l) \mbox{ $l$ even}: n/2,\quad H_4: 20 .$ \\
%The number of pairs of associate representations in $\Irr_{\gen}(\widetilde{W})$ for $\dim V_0$ odd:\\
%$A_{n-1}$: the number of partitions of $n$ into distinct parts with odd number of even parts,\\
%$B_n/C_n$: the number of partitions of $n$ of $n$, \\
%$D_n$: the number of partitions of $n$ with even number of parts \\
%$ E_7: 15,\quad H_3: 4.$
%\end{proof}

%\begin{remark}
%A deeper understanding on the map $i_S^{\pm}$ and the image $\im(\i_S^{\pm})$ for the crystallographic cases is studied in \cite{CT} with the technique in graded affine Hecke algebra.

%\end{remark}

\subsection{An orthonormal basis, and a connection to the graded affine Hecke algebra}

In this subsection, we will look at the example $W=W(I_2(n))$ ($n$ odd). We shall construct an orthonormal basis for $\overline{R}(W)$ such that every element in the basis is  the sum of irreducible characters in $\mathbb{Z}$-coefficients. The significance of this basis is explained in Remark \ref{rmk sig ob} below.

\bigskip

\noindent
{\it Example on $I_2(n)$, $n$ odd}

Let $W=W(I_2(n))$. It is easy to see that $\widetilde{W}'$ is a cyclic group of order $2n$. Then each genuine $\widetilde{W}$-representation $\widetilde{\rho}_i$ in Table \ref{tb ch odd} splits into two one-dimensional $\widetilde{W}'$-representations, denoted $\widetilde{\rho}^+_i$ and $\widetilde{\rho}^-_i$. We shall fix $\widetilde{\rho}^+_i$ and $\widetilde{\rho}^-_i$ such that \[ \widetilde{\rho}_i^+(f_{\alpha_1}f_{\alpha_2})=-e^{2i\sqrt{-1}\pi/n} \mbox{ and } \widetilde{\rho}_i^-(f_{\alpha_1}f_{\alpha_2})=-e^{-2i\sqrt{-1}\pi/n}. \]
The spin module $S$ of $\widetilde{W}$ is the two dimensional module $\widetilde{\rho}_{(n-1)/2}$. Set $S^+=\widetilde{\rho}_{(n-1)/2}^+$ and $S^-=\widetilde{\rho}_{(n-1)/2}^-$. Now a computation gives
\[      \sgn \otimes (S^+-S^-)= S^+-S^- , \]
and for $j=1, \ldots, (n-3)/2$,
\[      \phi_j \otimes (S^+-S^-) = \left(\widetilde{\rho}_{(n-1-2j)/2}^+-\widetilde{\rho}_{(n-1-2j)/2}^-\right)-\left(\widetilde{\rho}_{(n+1-2j)/2}^+-\widetilde{\rho}_{(n+1-2j)/2}^-\right) .\]
where $\sgn$ and $\phi_j$ are as in Table \ref{tb ch odd} and considered to be the restriction in $\widetilde{W}'$ (note that $\sgn$ restricted to $\widetilde{W}'$ is just the trivial representation). For notational convenience, set $\widetilde{\sigma}_{j}=\widetilde{\rho}_{(n+1-2j)/2}^+-\widetilde{\rho}_{(n+1-2j)/2}^-$ %$\widetilde{\sigma}_j=\left(\rho_{(n+1-2j)/2}^+-\rho_{(n+1-2j)/2}^-\right)$ 
for $j=1, \ldots, (n-1)/2$. 

By Proposition \ref{prop fix gp}, 
\[R_{\ind}(W)=\mathrm{span}_{\mathbb{C}}\left\{ \trivial \oplus \phi_1 \oplus \ldots \oplus \phi_{(n-1)/2}, \sgn\oplus \phi_1 \oplus \ldots \oplus \phi_{(n-1)/2} \right\} \]
 and hence $ \left\{  \sgn, \phi_1, \ldots, \phi_{(n-3)/2}  \right\}$
forms an ordered basis for $\overline{R}(W)$. We also fix an ordered basis for $\im(i^{\pm}_S)=\ker(1+\iota)$ (Proposition \ref{prop equality ik}):
\[  \left\{ \widetilde{\sigma}_{(n-1)/2},\ldots, \widetilde{\sigma}_1  \right\}.\] 

Then the matrix representation of $i_S^{\pm}$ with respect to the above two bases is an upper triangular matrix $I-J$, where $I$ is the $(n-1)/2 \times (n-1)/2$ identity matrix and $J$ is the $(n-1)/2 \times (n-1)/2$ matrix with $1$ in the superdiagonal and $0$ elsewhere. Then inspecting the columns of the matrix $(I-J)^{-1}=I+J+\ldots+J^{(n-3)/2}$, we could obtain a new basis $\left\{ \omega_i \right\}_{i=1}^{(n-1)/2}$ for $\overline{R}(W)$:
\[ \omega_1=\sgn,\quad  \omega_i = \sgn \oplus \phi_1 \oplus \ldots \oplus \phi_{i-1}\ (i=2, \ldots, (n-1)/2),
\]
 so that the new basis has the property that $i^{\pm}_S(\omega_i)=\widetilde{\sigma}_{i}$. Then, by Proposition \ref{prop isometry},
\[e_W(\omega_i, \omega_j)=\frac{1}{2} \langle i^{\pm}_S(\omega_i), i^{\pm}_S(\omega_j) \rangle_{\widetilde{W}'}=\langle \widetilde{\sigma}_i, \widetilde{\sigma}_j \rangle_{\widetilde{W}'}=\delta_{ij}.\]
Hence $\left\{ \omega_i \right\}_{i=1}^{(n-1)/2}$ is an orthonormal basis for $\overline{R}(W)$.

\begin{remark} \label{rmk sig ob}
The importance of this orthonormal basis is as follows. Let $\mathbb{H}$ be the graded affine Hecke algebra (see \cite{BCT} for the definition) associated to $W$ with the parameter function $c$. Let $X$ and $X'$ be discrete series of $\mathbb{H}$ (i.e. for the central character $\gamma$ of $X$ (or $X'$), $(\omega, \gamma)<0$ for all the fundamental weight $\omega$ in $R$). If $W$ is a Weyl group, then the discrete series as $W$-modules form an orthonormal set in $\overline{R}(W)$ (\cite[Section 3]{OS}, also see \cite{CT}). We conjecture the same holds in noncrystallographic cases and so the above constructed orthogonal basis would give some hints for the unknown $W$-module structure of discrete series in noncrystallographic cases. Some study for discrete series in noncrystallographic cases can be found in \cite{HO}, \cite{KR} and \cite{Kro}.

%In the case of $I_2(n)$ ($n$ odd), the orthogonal basis constructed above, in fact, agrees with the $W$-module structure of some tempered modules computed by Kriloff and Ram in an unpublished work.

\end{remark}

%then $e_W(X|_W, X'|_{W'})=0$ if $X \not\cong X'$, and $e_W(X|_W, X'|_W)=1$ if $X \cong X'$ 


\begin{thebibliography}{AFMO}



\bibitem{BCT} D. Barbasch, D. Ciubotaru and P. Trapa, {\it Dirac cohomology for graded affine Hecke algebras}, Acta. Math. {\bf 209} (2012), no. 2, 197-227.

\bibitem{BW} A. Borel and N. Wallach, {\it Continuous cohomology, discrete subgroups, and representations of reductive groups}, {\bf 94} (1980), Princeton University Press (Princeton, NJ).


\bibitem{Ca} R. Carter, {\it Finite groups of Lie type. Conjugacy classes and complex and characters}, Pure and Applied Mathematics (New York). A Wiley-Interscience Publication. John Wiley \& Sons, Inc. New York, 1985.


\bibitem{Ci} D. Ciubotaru, {\it Spin representations of Weyl groups and the Springer correspondence}, J. Reine Angew. Math. {\bf 671} (2012), 199-222.

\bibitem{CT} D. Ciubotaru and P. Trapa, {\it Characters of Springer representations on elliptic classes}, to appear in  Duke Math. J..

\bibitem{CM} D. H. Collingwood and W. M. McGovern, {\it Nilpotent orbits in semisimple Lie algebras}, Van Nostrand Reinhold, New York,
1985.


\bibitem{Mo1} A. O. Morris, {\it Projective representations of reflection groups. II}, Proc. London Math. Soc. (3) {\bf 40} (1980), no. 3, 553-576.

\bibitem{Mo2} A. O. Morris, {\it Projective characters of exceptional Weyl groups}, J. Algebra {\bf 29} (1974), 567-586.


\bibitem{Re} E. W. Read, {\it The linear and projective characters of the finite reflection group of type $H_4$}, Quart. J. Math. (Oxford), {\bf 25} (1974), 73-79.

\bibitem{Re2} E. W. Read, {\it On projective representations of the finite reflections of type $B_l$ and $D_l$}, J. London Math. Soc. (2), {\bf 10} (1975), 129-142.

\bibitem{Ree} M. Reeder, {\it Euler-Poincare pairings and elliptic representations of Weyl groups and p-adic groups}, Compositio Math. {\bf 129} (2001), 149-181.

%\bibitem{Slo} K. Slooten, Generalized Green functions and graded Hecke algebra, {\it Generalized Springer correspondence and Green functions for type $B/C$ graded Hecke algebras}, Adv. Math., {\bf 203} (2006), no.1, 34-108.
\bibitem{HO} G. J. Heckman and E. M. Opdam, {\it Yang's system of particles and Hecke algebras}, Math. Ann. {\bf 145} (1997), 139-173.
%\bibitem{HH} P. N. Hoffman and J. F. Humphreys, Projective representations of the symmetric groups, Oxford.

\bibitem{HP} J. S. Huang and P. Pandzic, {\it Dirac operators in representation theory}, Mathematics: Theory \& applications (2000), Birkhauser.

%\bibitem{KL} D. Kazhdan and G. Lusztig, {\it Proof of the Deligne-Langlands conjecture for affine Hecke algebras}, Invect. Math. {\bf 87} (1987), no. 1, 153-215.

\bibitem{Kro} C. Kriloff, {\it Some interesting nonspherical tempered representations of graded Hecke algebras}, Trans. Amer. Math. Soc. {\bf 351} (1999), 4411-4428.

\bibitem{KR} C. Kriloff, A. Ram, {\it Representations of graded Hecke algebras}, Represent. Theory {\bf 6} (2002), 31-69.



\bibitem{OS} E. Opdam and M. Solleveld, {\it Homological algebra for affine Hecke algebras}, Adv. Math. {\bf 220} (2009), 1549-1601.

\bibitem{St} J. Stembridge, {\it Shifted tableaux and the projective representations of symmetric groups}, Adv. Math. {\bf 74} (1989), 87-134.


\end{thebibliography}
\end{document}